\documentclass[11pt]{article}

\usepackage[a4paper,margin=1in]{geometry}
\usepackage{amsmath,amssymb,amsthm,mathtools}
\usepackage{array}
\usepackage{bm}
\usepackage{enumitem}
\usepackage{hyperref}
\usepackage{aliascnt}
\usepackage{tikz}
\usetikzlibrary{arrows.meta,positioning}

\hypersetup{
  colorlinks=true,
  linkcolor=blue!55!black,
  citecolor=green!45!black,
  urlcolor=blue!55!black,
  pdftitle={Residual neural networks overcome the curse of dimensionality for semilinear heat equations},
  pdfauthor={Ilkhom Mukhammadiev and Diyora Salimova},
  pdfsubject={Residual-network approximation of high-dimensional semilinear heat equations},
  pdfkeywords={residual neural networks, curse of dimensionality, semilinear heat equations, multilevel Picard approximations, neural-network expressivity}
}

\theoremstyle{plain}
\newtheorem{theorem}{Theorem}[section]

\newaliascnt{proposition}{theorem}
\newtheorem{proposition}[proposition]{Proposition}
\aliascntresetthe{proposition}

\newaliascnt{lemma}{theorem}
\newtheorem{lemma}[lemma]{Lemma}
\aliascntresetthe{lemma}

\newaliascnt{corollary}{theorem}
\newtheorem{corollary}[corollary]{Corollary}
\aliascntresetthe{corollary}

\theoremstyle{definition}
\newaliascnt{definition}{theorem}
\newtheorem{definition}[definition]{Definition}
\aliascntresetthe{definition}

\newaliascnt{setting}{theorem}
\newtheorem{setting}[setting]{Setting}
\aliascntresetthe{setting}

\newaliascnt{remark}{theorem}
\newtheorem{remark}[remark]{Remark}
\aliascntresetthe{remark}

\newaliascnt{assumption}{theorem}

\aliascntresetthe{assumption}

\usepackage{cleveref}
\crefname{theorem}{Theorem}{Theorems}
\Crefname{theorem}{Theorem}{Theorems}
\crefname{proposition}{Proposition}{Propositions}
\Crefname{proposition}{Proposition}{Propositions}
\crefname{lemma}{Lemma}{Lemmas}
\Crefname{lemma}{Lemma}{Lemmas}
\crefname{corollary}{Corollary}{Corollaries}
\Crefname{corollary}{Corollary}{Corollaries}
\crefname{definition}{Definition}{Definitions}
\Crefname{definition}{Definition}{Definitions}
\crefname{setting}{Setting}{Settings}
\Crefname{setting}{Setting}{Settings}
\crefname{remark}{Remark}{Remarks}
\Crefname{remark}{Remark}{Remarks}
\crefname{assumption}{Assumption}{Assumptions}
\Crefname{assumption}{Assumption}{Assumptions}

\newcommand{\R}{\mathbb{R}}
\newcommand{\N}{\mathbb{N}}
\newcommand{\Z}{\mathbb{Z}}
\newcommand{\E}{\mathbb{E}}
\newcommand{\PR}{\mathbb{P}}
\newcommand{\norm}[1]{\left\lVert#1\right\rVert}
\newcommand{\abs}[1]{\left\lvert#1\right\rvert}
\newcommand{\cR}{\mathcal{R}}
\newcommand{\cD}{\mathcal{D}}
\newcommand{\cP}{\mathcal{P}}
\newcommand{\cF}{\mathcal{F}}
\newcommand{\cU}{\mathcal{U}}
\newcommand{\bN}{\mathbf{N}}        
\newcommand{\bbN}{\bm{\mathcal{N}}} 
\newcommand{\Concat}{\ast}
\newcommand{\Lip}{\mathrm{Lip}}
\newcommand{\id}{\mathrm{Id}}
\newcommand{\Real}{\mathfrak{R}}
\newcommand{\Param}{\mathfrak{P}}
\newcommand{\Length}{\mathfrak{L}}

\title{\bfseries Residual neural networks overcome the curse of dimensionality for
semilinear heat equations}

\author{Ilkhom Mukhammadiev$^{1}$ and Diyora Salimova$^{2}$
	\bigskip\\
	\small{$^1$ Department of Applied Mathematics, University of Freiburg,}\vspace{-0.1cm}\\
	\small{Germany; e-mail: \href{mailto:ilkhom.mukhammadiev@mathematik.uni-freiburg.de}{\texttt{ilkhom.mukhammadiev@mathematik.uni-freiburg.de}}}\smallskip\\
	\small{$^2$ Department of Applied Mathematics, University of Freiburg,}\vspace{-0.1cm}\\
	\small{Germany; e-mail: \href{mailto:diyora.salimova@mathematik.uni-freiburg.de}{\texttt{diyora.salimova@mathematik.uni-freiburg.de}}}}
\date{\today}

\begin{document}
\maketitle
\vspace{-3em}

\begin{abstract}
\noindent
Rigorous results show that feedforward neural networks can overcome the curse of dimensionality
in the numerical approximation of high-dimensional partial differential equations (PDEs), but
comparatively little is known about residual neural networks (ResNets) in the nonlinear PDE
setting. We prove that ResNets overcome the curse of dimensionality in the numerical
approximation of solutions of semilinear heat equations with globally Lipschitz continuous,
gradient-independent nonlinearities: under polynomial growth and network approximability
hypotheses on the PDE data, there exist $\eta\in(0,\infty)$ and ResNets $\Psi_{d,\varepsilon}$,
$d\in\N$, $\varepsilon\in(0,1]$, with at most $\eta d^{\eta}\varepsilon^{-\eta}$ parameters
whose realizations approximate the solution in dimension $d$ with an $L^2$-error of at most
$\varepsilon$. The proof represents one deterministic realization of a multilevel Picard
estimator by a ResNet whose shortcut connections transmit the spatial variable and a scalar
accumulator, while the residual branches successively add the summands of the estimator. For
ridge-sum initial conditions, admissible sigmoidal activations, and globally Lipschitz
truncations of the nonlinearity, we obtain, for every $\xi>0$, the explicit bound
$C_\xi d^{4+\xi}\varepsilon^{-(3+\xi)}$ on the number of parameters.
\end{abstract}

\section{Introduction}\label{sec:intro}

\subsection{The problem and the main result}

Developing numerical methods for high-dimensional partial differential equations (PDEs) that avoid
the curse of dimensionality is among the central challenges of numerical analysis. Fix $T\in(0,\infty)$, a globally Lipschitz continuous function $f\colon\R\to\R$, and, for every
$d\in\N$, an initial value $g_d\colon\R^d\to\R$. We study residual neural network (ResNet) approximations of the solutions to the
semilinear heat equations
\begin{equation}\label{eq:pde}
  (\tfrac{\partial}{\partial t}u_d)(t,x)
  = (\tfrac12\Delta_x u_d)(t,x) + f\big(u_d(t,x)\big),
  \qquad u_d(0,x)=g_d(x),
\end{equation}
for $(t,x)\in(0,T)\times\R^d$. The curse of dimensionality refers here to the fact that, for
general high-dimensional approximation problems, the number of degrees of freedom employed by
standard approximation methods may grow exponentially in the PDE dimension $d$ and/or
in the reciprocal of the prescribed approximation accuracy (cf., e.g.,
Bellman~\cite{Bellman1957}, Novak \& Wo\'zniakowski~\cite{NovakWozniakowski2008}, and the
references mentioned therein).

The family of PDEs in~\eqref{eq:pde} includes, for example, equations with a globally Lipschitz
truncation of the Allen--Cahn nonlinearity $f(u)=u-u^3$; the untruncated cubic is not globally
Lipschitz continuous and is therefore not covered directly by our hypotheses (see
\Cref{sec:advantages} below). Additionally, a suitable globally Lipschitz truncation of the logistic reaction $r\mapsto r(1-r)$
 leads to a Fisher--KPP-type equation. After a time reversal, equations of the
form~\eqref{eq:pde} arise in nonlinear pricing, while more general semilinear Kolmogorov PDEs
appear in stochastic optimal control. In such applications the PDE dimension $d$ may be quite large.

The key contribution of this article is to prove that residual neural networks (ResNets) in the
sense of He et al.~\cite{HeZRS2016,HeZRS2016identity} approximate the solutions of~\eqref{eq:pde}
without the curse of dimensionality. More specifically, the main result of this article,
\Cref{thm:main} below, proves that, under suitable Lipschitz continuity, polynomial growth, and
network approximation hypotheses, there exist $\eta\in(0,\infty)$ and ResNets
$\Psi_{d,\varepsilon}$, $d\in\N$,
$\varepsilon\in(0,1]$, such that $\Real_a(\Psi_{d,\varepsilon})\in C(\R^d,\R)$ and
\[
   \Param(\Psi_{d,\varepsilon})\le \eta\,d^{\eta}\,\varepsilon^{-\eta}
   \qquad\text{and}\qquad
   \Big[\textstyle\int_{[0,1]^d}\abs{u_d(T,x)-(\Real_a(\Psi_{d,\varepsilon}))(x)}^2\,dx\Big]^{1/2}\le \varepsilon,
\]
where $\Param$ denotes the number of parameters of a ResNet and $\Real_a$ denotes its
realization with respect to the activation function $a\colon\R\to\R$.

Since the real number $\eta$ depends neither on $d$ nor on $\varepsilon$, the number of parameters
used to describe the ResNets $\Psi_{d,\varepsilon}$ grows at most polynomially in both the PDE
dimension $d$ and the reciprocal $\varepsilon^{-1}$ of the prescribed approximation accuracy. In
this sense, the ResNets $\Psi_{d,\varepsilon}$ overcome the curse of dimensionality in the
numerical approximation of the PDEs in~\eqref{eq:pde}. We emphasize that, as in the related
results in the scientific literature, \Cref{thm:main} is a deterministic expressivity result. It
establishes the existence of a ResNet with the stated approximation properties, but it does not
assert that a prescribed training procedure computes such a ResNet.

\subsection{Deep learning for high-dimensional PDEs and rigorous approximation results}\label{sec:rigorous}

In recent years, neural networks have been successfully employed for the numerical approximation
of solutions to high-dimensional PDEs. For example, the deep BSDE method of E et
al.~\cite{EHanJentzen2017,HanJentzenE2018} reformulates a semilinear parabolic PDE as a backward
stochastic differential equation and approximates its solution by means of stochastic gradient
descent in dimensions of order $100$. Further deep-learning-based approximation methods for PDEs
include, for instance, the physics-informed neural networks of Raissi et
al.~\cite{RaissiPerdikarisKarniadakis2019}, the deep Galerkin method of Sirignano \&
Spiliopoulos~\cite{SirignanoSpiliopoulos2018}, and the deep Ritz method of E \&
Yu~\cite{EYu2018}. Numerical studies for Kolmogorov equations can, e.g., be found in Beck et
al.~\cite{BeckBeckerGrohs2021}, and we refer to~\cite{BeckHutzenthalerJentzenKuckuck2023} for an
overview. These numerical simulations indicate that neural networks possess the flexibility
required to approximate solutions of certain high-dimensional PDEs. They do not, however, by
themselves establish polynomial bounds for the number of parameters used to describe the
approximating networks.

There are now several rigorous results in the scientific literature which prove that neural
networks possess the requisite expressive power for various classes of PDEs. The first results of
this type concerned linear equations: polynomial parameter bounds for Black--Scholes and
Kolmogorov equations~\cite{GHJV2018,ElbrachterGrohsJentzenSchwab2022,JentzenSalimovaWelti2021},
including the generalization error~\cite{BernerGrohsJentzen2020}, space--time
approximations~\cite{HornungJentzenSalimova2024}, option prices under exponential L\'evy
models~\cite{GononSchwab2021}, and the Poisson equation with Dirichlet boundary
conditions~\cite{GrohsHerrmann2020}; Reisinger \& Zhang~\cite{ReisingerZhang2020} treat nonsmooth
value functions of zero-sum games of nonlinear stiff systems.

For nonlinear equations, one of the principal tools is the class of multilevel Picard (MLP)
approximations introduced by E et al.~\cite{EHJK2019,EHJK2021smooth} and further developed
in~\cite{HJKNW2018,HutzenthalerKruse2020grad,BHHJK2019AllenCahn,BeckerBraunwarth2020,
HutzenthalerJentzenKruseNguyen2020fcm}. MLP approximations are full-history recursive nonlinear
Monte Carlo schemes which, for the problem classes and assumptions treated in these works,
approximate solutions of semilinear parabolic PDEs at a computational cost growing at most
polynomially in the PDE dimension $d$ and the reciprocal $\varepsilon^{-1}$ of the prescribed
approximation accuracy. A principal analytic ingredient of
the present article is the result of Hutzenthaler et al.~\cite{HJKN2020}. The authors prove that
rectified feedforward neural networks approximate solutions of semilinear heat equations of the
form~\eqref{eq:pde} without the curse of dimensionality by showing that one realization of the
MLP estimator can be represented exactly by a ReLU network.
Truncated MLP
approximations for the Allen--Cahn equation are developed in~\cite{BHHJK2019AllenCahn}, and gradient-dependent
nonlinearities in the MLP setting are treated by Hutzenthaler \& Kruse~\cite{HutzenthalerKruse2020grad}; the latter extends MLP convergence to nonlinearities
$f=f(u,\nabla u)$ through a stochastic fixed-point system involving the Bismut--Elworthy--Li
formula. The corresponding network-representation theorem for gradient-dependent
semilinear heat equations in the feedforward setting was obtained by Neufeld \& Nguyen~\cite{NeufeldNguyen2024}; analogous results for general semilinear PDEs with gradient-dependent
nonlinearities and nonconstant coefficients are due to Neufeld et al.~\cite{NeufeldNguyenWu2023} at the MLP level.

The class of PDEs covered by the network-representation theorems has been extended in several
directions: to general semilinear Kolmogorov PDEs with gradient-independent Lipschitz
nonlinearities and Lipschitz drift and diffusion coefficients by Cioica-Licht et
al.~\cite{CioicaHutzenthalerWerner2022}; to $L^p$-errors for arbitrary $p\in(0,\infty)$ and to
further activations, including ReLU, leaky ReLU, and softplus, by Ackermann et
al.~\cite{AckermannJentzenKruseKuckuckPadgett2023} and Neufeld \&
Nguyen~\cite{NeufeldNguyen2026MLP}, and to space--time domains
in~\cite{AckermannJentzenKuckuckPadgett2024}; to partial integro-differential equations (PIDEs)
for jump processes, in the linear case by Gonon \& Schwab~\cite{GononSchwab2021PIDE} and in the
semilinear case by Neufeld \& Wu~\cite{NeufeldWu2022} via an MLP scheme, with the corresponding
expressivity result in~\cite{NeufeldNguyenWu2023PIDE} and a random deep splitting method
in~\cite{NeufeldSchmockerWu2024}; to random feature networks for Black--Scholes and exponential
L\'evy models by Gonon~\cite{Gonon2021Random}; and to $L^\infty$-approximation for the linear heat
equation by Gonon et al.~\cite{GononGrohsJentzenKoflerSiska2022}.

The upper bounds described above are complemented by lower bounds which identify function
classes for which the curse of dimensionality cannot be overcome by network approximations with a given
architecture. Grohs et al.~\cite{GrohsIbragimovJentzenKoppensteiner2023}
exhibit a class of high-dimensional target functions for which sufficiently deep networks overcome
the curse of dimensionality but shallow networks (in particular, single-hidden-layer networks) do
not; the depth required for the polynomial parameter bound grows in the dimension. The present
result is consistent with this lower bound, since the ResNets $\Psi_{d,\varepsilon}$ produced by
\Cref{thm:main} are not of uniformly bounded depth. More specifically, the accumulator
construction follows the finite MLP recursion tree, whose size grows with $d$ and
$\varepsilon^{-1}$ through the balancing in \Cref{lem:mlp-cost} below. A further sampling-complexity
lower bound of Grohs \& Voigtl\"ander~\cite{GrohsVoigtlaender2024} shows that, on suitable
neural-network approximation spaces, favorable approximation rates cannot in general be realized
by deterministic or randomized algorithms using only polynomially many point samples; in the
uniform norm, the worst-case sampling complexity suffers from the curse of dimensionality. This
result does not contradict our
upper bound, because \Cref{thm:main} is an expressivity result and not a learning result. It does,
however, show that the polynomial existence guarantee in \Cref{thm:main} alone does not imply the
existence of a polynomial-time training algorithm.

The classical expressivity theory for feedforward
networks~\cite{HornikStinchcombeWhite1989,Cybenko1989,LLPS1993,Mhaskar1996,Pinkus1999} and the
quantitative ReLU rate bounds of~\cite{Yarotsky2017,PetersenVoigtlaender2018} provide the general
background for the staircase construction in \Cref{lem:1d-approx}.

A recent result of Yang \& Pan~\cite{YangPan2026} studies ResNet approximations of
high-dimensional stochastic PDEs by means of a splitting-up method and repeated cross-iteration of
two ResNets; its stochastic-PDE setting and network construction differ from the deterministic
MLP accumulator developed here. To the best of our knowledge, the present article provides the
first expressivity result for
residual networks, with parameter counts polynomial in $d$ and $\varepsilon^{-1}$, for the
semilinear heat equations~\eqref{eq:pde}. The feedforward results in~\cite{HJKN2020} and
\cite{AckermannJentzenKruseKuckuckPadgett2023} treat the same gradient-independent semilinear heat
equation, but they require an activation function which represents the identity (such as ReLU,
leaky ReLU, or softplus). The residual construction developed below removes this requirement and
yields explicit polynomial exponents. In comparison,
\cite{NeufeldNguyen2024} treats the more general gradient-dependent nonlinearity
$f(u,\nabla u)$ in the feedforward setting, whereas the present article considers a
gradient-independent nonlinearity and a residual architecture with a different class of admissible
activation functions, including bounded monotone sigmoidal activations.

\subsection{Residual networks and their approximation theory}\label{sec:resnets}

Residual neural networks were introduced by He et
al.~\cite{HeZRS2016,HeZRS2016identity} and are characterized by the presence of shortcut
connections. More specifically, in the case of an identity shortcut, the $i$th residual update
has the form
\[
   x_i=x_{i-1}+F_i(x_{i-1}),
\]
where $F_i$ is the realization of the $i$th residual branch; the recurrence can be interpreted as
one step of an explicit Euler discretization of an ordinary differential equation, an
interpretation developed in~\cite{E2017,HaberRuthotto2017,LuZhongLiDong2018,ChenRubanova2018}. Rigorous approximation results for ResNets include the universal
approximation theorem of Lin \& Jegelka~\cite{LinJegelka2018} and the control-theoretic analysis
of Tabuada \& Gharesifard~\cite{TabuadaGharesifard2023}. A principal architectural ingredient in
this article is the result of Baggenstos \& Salimova~\cite{BaggenstosSalimova2023}, in which the
skip-connection recurrence is identified with an Euler--Maruyama discretization and ResNets are
shown to overcome the curse of dimensionality for a class of linear Kolmogorov PDEs.

In this article we employ the residual architecture to construct a ResNet which successively
accumulates the summands occurring in an MLP approximation. More specifically, each shortcut
transmits the spatial variable together with the current value of a scalar accumulator, while each
accumulator update adds one summand of the MLP formula; depending on the summand, an update
consists of a single residual block or a short concatenation of such blocks. The composition rule in
\Cref{lem:resnet-comp}(ii) then implies that the parameter count of the resulting ResNet grows
additively under concatenation. This construction requires neither an FNN representation of
the identity nor a parallelization of networks with different widths. Consequently, in contrast
to the feedforward constructions as in~\cite{HJKN2020}, we do not impose the hypothesis that the
activation function represents the identity. Under the abstract surrogate hypothesis in
\Cref{set:nn}, part~\textup{(i)} of \Cref{thm:main} applies to every activation for which the required data surrogates
exist. The explicit constructions in \Cref{cor:trunc-surrogate,lem:ridge-surrogate} show that the
admissible activations in part~\textup{(ii)} include the admissible sigmoidal activations of
\Cref{lem:1d-approx}, such as the logistic function and $\tanh$.

Roughly speaking, our proof of \Cref{thm:main} consists of the following steps. First, we
employ the stability and complexity estimates for MLP approximations in
\Cref{cor:mlp-stab,lem:mlp-cost} to approximate the solution of~\eqref{eq:pde} by an MLP estimator
built from neural-network surrogates of the PDE data. This allows us to select one deterministic realization of this estimator
with the desired error. Second, the accumulator construction in \Cref{prop:repr} represents this
realization by a ResNet, and the recursion in \Cref{lem:prec} proves the required polynomial bound
for the number of parameters. In addition, \Cref{lem:ridge-surrogate} provides a concrete class of
ridge-sum initial conditions which satisfies the abstract data hypotheses.

\section{A calculus of feedforward and residual networks}\label{sec:calculus}

In this section we recall, following~\cite{BaggenstosSalimova2023,JentzenSalimovaWelti2021}, the
formalism for feedforward and residual networks used throughout this article and the auxiliary
results needed below, notably the additivity of ResNet composition with respect to the number of
parameters (\Cref{lem:resnet-comp}).

\begin{samepage}
\paragraph{Notation conventions.}
The following symbols are used throughout \Cref{sec:calculus} and the subsequent sections.
{\small
\begin{center}
\renewcommand{\arraystretch}{1.18}
\begin{tabular}{@{}lp{11.0cm}@{}}
\hline
$\N_0$ & $\N \cup \{0\}$ \\
$\norm{\cdot}$ & the Euclidean norm; $\norm{x}=(\sum_{i=1}^d x_i^2)^{1/2}$ for $x \in \R^d$\\
$\id_n$ & the $n\times n$ identity matrix \\
$\bN$ & the set of feedforward neural networks (FNNs); see \Cref{def:fnn}\\
$\bbN=\bigcup_{n\in\N}\bbN_n$ & the set of residual neural networks (ResNets); $\bbN_n$ is the subset of ResNets of length $n$; see \Cref{def:resnet}\\
$\theta$ & a generic FNN, $\theta=((W_1,B_1),\dots,(W_L,B_L))\in\bN$\\
$\Theta$ & a generic ResNet, $\Theta=(\Gamma_1,\theta_1,\dots,\Gamma_n,\theta_n)\in\bbN$\\
$\mathcal L(\theta)$ & the depth of an FNN; see \Cref{def:fnn}\\
$\theta_i$ & the $i$th residual branch (an FNN) of a ResNet\\
$(\Gamma_i,\theta_i)$ & the $i$th residual block of a ResNet\\
$\Gamma_i$ & the $i$th shortcut matrix of a ResNet\\
$\mathcal{I}(\cdot),\mathcal{O}(\cdot)$ & input and output dimensions (of an FNN or a ResNet); see \Cref{def:fnn,def:resnet}\\
$\cD(\cdot)$ & the architecture (of an FNN or a ResNet); see \Cref{def:fnn,def:resnet}\\
$\Length(\Theta)$ & the length of a ResNet, i.e.\ the number of residual blocks; see \Cref{def:resnet}\\
$\cP(\theta)$ & the parameter count of an FNN; see \eqref{eq:fnn_p}\\
$\Param(\Theta)$ & the parameter count of a ResNet; see~\eqref{eq:resnet-P}\\
$\cR_a(\theta)$ & the realization of an FNN $\theta$ with activation $a$; see \Cref{def:fnn}\\
$\Real_a(\Theta)$ & the realization of a ResNet $\Theta$ with activation $a$; see \Cref{def:resnet-real}\\
$\Theta_2\Concat\Theta_1$ & the concatenation of two ResNets; see \Cref{lem:resnet-comp}\\
$\mathcal{M}_a$ & the componentwise action of $a$, i.e.\ $\mathcal{M}_a(x_1,\dots,x_n)=(a(x_1),\dots,a(x_n))$\\
\hline
\end{tabular}
\end{center}}
\end{samepage}

\subsection{Feedforward neural networks (FNNs)}

\begin{definition}[FNNs]\label{def:fnn}
Let
\[
  \bN=\bigcup_{L\in\N}\ \bigcup_{l_0,l_1,\dots,l_L\in\N}
  \Big(\bigtimes_{k=1}^{L}\big(\R^{l_k\times l_{k-1}}\times\R^{l_k}\big)\Big).
\]
For $\theta=((W_1,B_1),\dots,(W_L,B_L))\in \bigtimes_{k=1}^{L}\big(\R^{l_k\times l_{k-1}}\times\R^{l_k}\big) \subset \bN$ define the \emph{depth}
$\mathcal{L}(\theta)=L$, the \emph{input/output dimensions} $\mathcal{I}(\theta)=l_0$, $\mathcal{O}(\theta)=l_L$, the
\emph{architecture} $\cD(\theta)=(l_0,l_1,\dots,l_L)$, and the \emph{number of parameters}
\begin{equation}\label{eq:fnn_p}
   \cP(\theta)=\sum_{k=1}^{L} l_k\,(l_{k-1}+1).
\end{equation}
Given an activation function $a\in C(\R,\R)$, write
$\mathcal{M}_a(y_1,\dots,y_n):=(a(y_1),\dots,a(y_n))$ for its componentwise action, and define
$x_k:=\mathcal{M}_a(W_k x_{k-1}+B_k)$ for $k\in\N\cap(0,L)$ recursively from $x_0\in\R^{l_0}$. The
\emph{realization} function $\cR_a(\theta)\in C(\R^{l_0},\R^{l_L})$ of $\theta$ is then given by
$(\cR_a\theta)(x_0):=W_L x_{L-1}+B_L$.
\end{definition}

\subsection{Residual neural networks (ResNets)}

\begin{definition}[ResNets]\label{def:resnet}
The set of ResNets is $\bbN=\bigcup_{n\in\N}\bbN_n$, where for $n\in\N$
\[
  \bbN_n=\bigcup_{(d_0,\dots,d_n)\in\N^{n+1}}
  \Big\{(\Gamma_1,\theta_1,\dots,\Gamma_n,\theta_n):
  \theta_k\in\bN,\ \mathcal{I}(\theta_k)=d_{k-1},\ \mathcal{O}(\theta_k)=d_k,\ \Gamma_k\in\R^{d_k\times d_{k-1}}\Big\}.
\]
Given a ResNet $\Theta=(\Gamma_1,\theta_1,\dots,\Gamma_n,\theta_n) \in \bbN$,
we call $\theta_1,\dots,\theta_n$ its \emph{residual branches}, $\Gamma_1,\dots,\Gamma_n$ its
\emph{shortcut} (skip-connection) matrices, and $(\Gamma_i,\theta_i)$, $i\in\{1,\dots,n\}$, its
\emph{residual blocks}. We define its \emph{length} as
$\Length(\Theta)=n$, its input and output dimensions as $\mathcal I(\Theta)=d_0$ and
$\mathcal O(\Theta)=d_n$, its architecture as $\cD(\Theta)=(d_0,\dots,d_n)$, and its
\emph{parameter count} as
\begin{equation}\label{eq:resnet-P}
  \Param(\Theta)=\sum_{i=1}^{n}\big(\cP(\theta_i)+d_i d_{i-1}\big).
\end{equation}
\end{definition}

\noindent The above formulation follows \cite[Definition~3.1]{BaggenstosSalimova2023}.

\begin{definition}[Realization of a ResNet]\label{def:resnet-real}
Let $a\in C(\R,\R)$, $n \in \N$, and let $\Theta=(\Gamma_i,\theta_i)_{i=1}^{n}\in\bbN$, with residual branches
$\theta_i$ and shortcut matrices $\Gamma_i$ (cf.\ \Cref{def:resnet}). Its realization
$\Real_a(\Theta)$ is the map $\R^{\mathcal{I}(\theta_1)} \ni x_0\mapsto x_n \in \R^{\mathcal{O}(\theta_n)}$ defined by the recurrence
\begin{equation}\label{eq:skip}
  x_i=\Gamma_i\,x_{i-1}+(\cR_a\theta_i)(x_{i-1}),\qquad \forall\, i\in\{1,\dots,n\}.
\end{equation}
\end{definition}

\noindent This formulation of the realization recurrence follows
\cite[Definition~3.2]{BaggenstosSalimova2023}.

When $d_0=\dots=d_n$ and $\Gamma_i=\id_{d_0}$ for all $i$, \eqref{eq:skip} has the additive
form $x_i=x_{i-1}+(\cR_a\theta_i)(x_{i-1})$ of an explicit Euler-type update. Such updates
encompass Euler--Maruyama steps when the residual branch represents a drift-plus-diffusion
increment with a fixed Brownian increment; see~\cite{BaggenstosSalimova2023}.
\Cref{fig:resnet} illustrates
the structure.

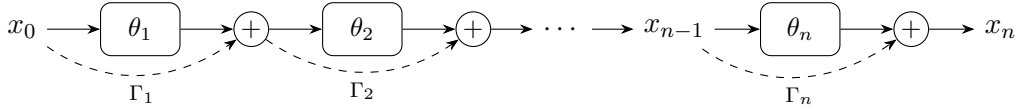
\begin{figure}[ht]
\centering
\begin{tikzpicture}[
   box/.style={draw,rounded corners,minimum width=1.05cm,minimum height=0.7cm,font=\small},
   sum/.style={draw,circle,inner sep=1pt,font=\small},
   >=Stealth, node distance=1.05cm]
  \node (x0) {$x_0$};
  \node[box,right=0.7 of x0] (t1) {$\theta_1$};
  \node[sum,right=0.7 of t1] (s1) {$+$};
  \node[box,right=0.7 of s1] (t2) {$\theta_2$};
  \node[sum,right=0.7 of t2] (s2) {$+$};
  \node[right=0.55 of s2] (dots) {$\cdots$};
  \node[right=0.55 of dots] (xn1) {$x_{n-1}$};
  \node[box,right=0.6 of xn1] (tn) {$\theta_n$};
  \node[sum,right=0.7 of tn] (sn) {$+$};
  \node[right=0.6 of sn] (out) {$x_n$};
  \draw[->] (x0)--(t1); \draw[->] (t1)--(s1);
  \draw[->] (s1)--(t2); \draw[->] (t2)--(s2);
  \draw[->] (s2)--(dots); \draw[->] (dots)--(xn1); \draw[->] (xn1)--(tn);
  \draw[->] (tn)--(sn); \draw[->] (sn)--(out);
  \draw[->,dashed] (x0) to[out=-35,in=-145] node[below,font=\scriptsize]{$\Gamma_1$} (s1);
  \draw[->,dashed] (s1) to[out=-35,in=-145] node[below,font=\scriptsize]{$\Gamma_2$} (s2);
  \draw[->,dashed] (xn1) to[out=-35,in=-145] node[below,font=\scriptsize]{$\Gamma_n$} (sn);
\end{tikzpicture}
\caption{Realization of a ResNet. The boxes labeled $\theta_i$ represent the residual branches, and
the dashed arrows represent the linear shortcuts $\Gamma_i$. If $\Gamma_i=\id$, then the
recurrence has the additive structure of an explicit Euler-type update.}
\label{fig:resnet}
\end{figure}

\begin{lemma}[Composition of ResNets]\label{lem:resnet-comp}
Let $a\in C(\R,\R)$, $m, n \in \N$, and let
$\Theta_1=(\Gamma^1_1,\theta^1_1,\dots, \allowbreak\Gamma^1_n,\theta^1_n)$ and
$\Theta_2=(\Gamma^2_1,\theta^2_1,\dots,\Gamma^2_m,\theta^2_m)$ be ResNets with
$\mathcal{O}(\theta^1_n)=\mathcal{I}(\theta^2_1)$. We define their composition by concatenation,
\[
  \Theta_2\Concat\Theta_1
  =(\Gamma^1_1,\theta^1_1,\dots,\Gamma^1_n,\theta^1_n,
    \Gamma^2_1,\theta^2_1,\dots,\Gamma^2_m,\theta^2_m)\in\bbN.
\]
Then
\[
  \text{(i)}\quad \Real_a(\Theta_2\Concat\Theta_1)=\Real_a(\Theta_2)\circ\Real_a(\Theta_1),
  \qquad\qquad
  \text{(ii)}\quad \Param(\Theta_2\Concat\Theta_1)=\Param(\Theta_1)+\Param(\Theta_2).
\]
\end{lemma}

\begin{proof}[Proof of \Cref{lem:resnet-comp}]
The assertions follow from \cite[Definition~3.3 and Lemma~3.4]{BaggenstosSalimova2023}. For
completeness we note that the recurrence~\eqref{eq:skip} on the concatenated ResNet first runs
through the blocks of $\Theta_1$, reaching $\Real_a(\Theta_1)(x_0)$, and then through the blocks
of $\Theta_2$, while~\eqref{eq:resnet-P} adds the two parameter counts.
This proves assertions~\textup{(i)}--\textup{(ii)}.
\end{proof}
\begin{lemma}[One-coordinate lift]\label{lem:resnet-lift}
Let $a\in C(\R,\R)$ and let
$\Theta=((\Gamma_i,\theta_i))_{i=1}^{n}\in\bbN_n$ have architecture
$\cD(\Theta)=(d_0,\dots,d_n)$. Then there exists a ResNet
$\operatorname{Lift}(\Theta)\in\bbN_n$ with the architecture
$(d_0+1,\dots,d_n+1)$ such that
\begin{equation}\label{eq:resnet-lift}
  \Real_a(\operatorname{Lift}(\Theta))(x,s)
  =\big(\Real_a(\Theta)(x),s\big)
  \qquad (x\in\R^{d_0},\ s\in\R),
\end{equation}
and
\begin{equation}\label{eq:resnet-lift-cost}
  \Param(\operatorname{Lift}(\Theta))\le 4\,\Param(\Theta).
\end{equation}
No assumption on $a$ beyond continuity is required.
\end{lemma}

\begin{proof}[Proof of \Cref{lem:resnet-lift}]
For each $i\in\{1,\dots,n\}$, let $\theta_i^\uparrow$ be an FNN of input dimension
$d_{i-1}+1$ and output dimension $d_i+1$ whose realization is
\[
   (\cR_a\theta_i^\uparrow)(z,s)=\big((\cR_a\theta_i)(z),0\big).
\]
If $\theta_i$ has depth one, this is obtained by appending one zero column and one zero row to
its weight matrix and one zero coordinate to its bias vector; the parameter count then satisfies
$\cP(\theta_i^\uparrow)=(d_i+1)(d_{i-1}+2)\le3d_i(d_{i-1}+1)=3\cP(\theta_i)$, using
$d_i,d_{i-1}\ge1$. If its depth $L$ is at least two, write
$\cD(\theta_i)=(d_{i-1},m_1,\dots,m_{L-1},d_i)$, adjoin one zero
column to the first weight matrix and one zero row and bias coordinate to the last affine layer, and leave
all intermediate layers unchanged; this increases the count by
$m_1+(m_{L-1}+1)\le2\cP(\theta_i)$, since $m_1\le\cP(\theta_i)$ and
$m_{L-1}+1\le\cP(\theta_i)$. In either case
$\cP(\theta_i^\uparrow)\le 3\cP(\theta_i)$. Let
\[
   \Gamma_i^\uparrow=
   \begin{pmatrix}
      \Gamma_i&0\\[1pt]
      0&1
   \end{pmatrix}
   \in\R^{(d_i+1)\times(d_{i-1}+1)}.
\]
Since $d_i,d_{i-1}\ge1$, the shortcut count satisfies
$(d_i+1)(d_{i-1}+1)\le4d_i d_{i-1}$. The recurrence~\eqref{eq:skip} for
$((\Gamma_i^\uparrow,\theta_i^\uparrow))_{i=1}^n$ leaves the last coordinate unchanged and runs the
original recurrence in the first coordinates, which proves~\eqref{eq:resnet-lift}. Summing the
blockwise bounds in~\eqref{eq:resnet-P} gives~\eqref{eq:resnet-lift-cost}.
\end{proof}

\section{Semilinear heat equations and MLP approximations}\label{sec:mlp}

In this section we recall, in the form required below, the MLP approximations
of~\cite{HJKN2020,HJKNW2018,EHJK2019}. Following the organization of
\cite[Section~2]{HJKN2020}, we collect the data, the probability space, the stochastic fixed-point
equation, and the MLP estimator in \Cref{set:mlp}. Every result in this section is formulated in
this setting. More specifically, \eqref{eq:sfp} below records the stochastic fixed-point
representation of the considered PDE solution, and \eqref{eq:mlp} defines the approximation which
will subsequently be represented by a ResNet. The two quantitative ingredients used in the proof
of \Cref{thm:main} are the stability estimate in \Cref{cor:mlp-stab} and the complexity estimate in
\Cref{lem:mlp-cost} below.

\begin{setting}[Semilinear heat equation and its MLP approximation, cf. {\cite[Setting~2.1]{HJKN2020}}]\label{set:mlp}
Let $d\in\N$, $T,L,B\in(0,\infty)$, $p\in[1,\infty)$, $q\in[2,\infty)$, and $\Delta\in(0,1]$. Let
$g_1,g_2\in C(\R^d,\R)$ and $f_1,f_2\in C(\R,\R)$ satisfy, for all $v,w\in\R$ and all $x\in\R^d$, the
Lipschitz bounds and growth bounds
\begin{equation}\label{eq:data-bounds}
  \abs{f_i(w)-f_i(v)}\le L\abs{w-v},\qquad
  \max\{\abs{f_i(0)},\abs{g_i(x)}\}\le B(1+\norm{x})^p,\qquad i\in\{1,2\},
\end{equation}
and the perturbation bound
\begin{equation}\label{eq:data-pert}
  \max\{\abs{f_1(v)-f_2(v)},\,\abs{g_1(x)-g_2(x)}\}\le\Delta\big((1+\norm{x})^{pq}+\abs{v}^q\big).
\end{equation}
Here $(g_1,f_1)$ represents the original PDE data, whereas $(g_2,f_2)$ represents perturbed data
which will later be chosen as network realizations. If $(g_2,f_2)=(g_1,f_1)$, then the perturbation bound holds for every
$\Delta\in(0,1]$. Let $\mathfrak T=\bigcup_{n\in\N}\Z^n$, and let $(\Omega,\cF,\PR)$ be a
probability space which supports a standard $d$-dimensional Brownian motion $\mathbf W$, i.i.d.\
random variables $\mathfrak u^\vartheta$, $\vartheta\in\mathfrak T$, which are uniformly distributed on
$[0,1]$, and independent standard $d$-dimensional Brownian motions $W^\vartheta$,
$\vartheta\in\mathfrak T$. Assume that $\mathbf W$, $(\mathfrak u^\vartheta)_{\vartheta\in\mathfrak T}$, and
$(W^\vartheta)_{\vartheta\in\mathfrak T}$ are mutually independent. For every $\vartheta\in\mathfrak T$ and
$t\in[0,T]$, let
$\cU_t^\vartheta=t+(T-t)\mathfrak u^\vartheta$.

For $i\in\{1,2\}$ let $v_i\in C([0,T]\times\R^d,\R)$ denote the unique at most polynomially
growing solution of the stochastic fixed-point equation: for all $s\in[0,T]$, $x\in\R^d$,
\begin{equation}\label{eq:sfp}
  v_i(s,x)=\E\!\left[g_i\big(x+\mathbf{W}_{T-s}\big)
         +\int_s^T f_i\big(v_i(t,x+\mathbf{W}_{t-s})\big)\,dt\right].
\end{equation}
Here \emph{at most polynomially growing} means that there exist $c,r\in(0,\infty)$, possibly
depending on $i$ and $d$, such that $\abs{v_i(s,x)}\le c(1+\norm{x})^{r}$ for all
$(s,x)\in[0,T]\times\R^d$; existence and uniqueness in this class are recalled after this
setting. Whenever the semilinear heat equation~\eqref{eq:pde} with data $(g_i,f_i)$ possesses an
at most polynomially growing solution $\mathsf u_i$ in the class
\begin{equation}\label{eq:reg-class}
  \mathsf u_i\in C([0,T]\times\R^d,\R)\cap C^{1,2}((0,T]\times\R^d,\R),
\end{equation}
the Feynman--Kac/Duhamel principle
(with $\mathbf W$ the standard Brownian motion of generator $\tfrac12\Delta$, so that no rescaling is
needed) shows that the time reversal $(s,x)\mapsto\mathsf u_i(T-s,x)$
solves~\eqref{eq:sfp}; by uniqueness,
\begin{equation}\label{eq:reversal}
  v_i(s,x)=\mathsf u_i(T-s,x)\qquad(s\in[0,T],\ x\in\R^d).
\end{equation}
In particular $v_i(T,\cdot)=g_i$ and, whenever $\mathsf u_i$ exists,
$v_i(0,\cdot)=\mathsf u_i(T,\cdot)$ is the time-$T$ solution that the main theorem approximates.
The regularity class~\eqref{eq:reg-class} allows for merely continuous
initial data $g_i$ that are smoothed instantaneously by the heat flow.

For $i\in\{1,2\}$, $M\in\N$, $\vartheta\in\mathfrak T$, and
$n\in\N_0\cup\{-1\}$ define the multilevel Picard estimator
$U^{i,\vartheta}_{n,M}\colon[0,T]\times\R^d\times\Omega\to\R$ by
$U^{i,\vartheta}_{-1,M}=U^{i,\vartheta}_{0,M}=0$ and, for $n\in\N$,
\begin{equation}\label{eq:mlp}
\begin{aligned}
  U^{i,\vartheta}_{n,M}(t,x)
  &=\frac{1}{M^n}\sum_{j=1}^{M^n} g_i\big(x+W^{(\vartheta,0,-j)}_{T}-W^{(\vartheta,0,-j)}_{t}\big)\\
  &\quad+\sum_{l=0}^{n-1}\frac{T-t}{M^{\,n-l}}
    \sum_{j=1}^{M^{\,n-l}}
    \Big[\big(f_i\circ U^{i,(\vartheta,l,j)}_{l,M}\big)-\mathbf{1}_{\N}(l)\big(f_i\circ U^{i,(\vartheta,-l,j)}_{l-1,M}\big)\Big]\\
  &\hphantom{\quad+\sum_{l=0}^{n-1}\frac{T-t}{M^{\,n-l}}\sum_{j=1}^{M^{\,n-l}}}
    {}\times\Big(\cU^{(\vartheta,l,j)}_t,\,x+W^{(\vartheta,l,j)}_{\cU^{(\vartheta,l,j)}_t}-W^{(\vartheta,l,j)}_t\Big).
\end{aligned}
\end{equation}
The estimator $U^{i,\vartheta}_{n,M}(t,\cdot)$ approximates the fixed-point solution $v_i(t,\cdot)$; in
particular $U^{i,0}_{N,M}(0,\cdot)$ approximates $v_i(0,\cdot)$.
We write $v:=v_1$, $U^\vartheta_{n,M}:=U^{1,\vartheta}_{n,M}$, $g:=g_1$, and $f:=f_1$ when only the true
data are involved, and we retain the notation $U^{2,\vartheta}_{n,M}$ for the estimator associated
with the surrogate data $(g_2,f_2)$.
\end{setting}

\noindent
For globally Lipschitz nonlinearities $f_i$ and at most polynomially growing $g_i$, existence and
uniqueness of the at most polynomially growing continuous solution $v_i$ of~\eqref{eq:sfp} follow
from \cite[Corollary~3.10]{BeckGononHutzenthalerJentzen2021}; for the classical parabolic
background and the Feynman--Kac representation used in~\eqref{eq:reversal} see
Friedman~\cite{Friedman1964} for
the linear theory, Pardoux \& Peng~\cite{PardouxPeng1992} for the connection to backward
stochastic differential equations as a nonlinear Feynman--Kac formula, and
\cite[Section~2]{HJKN2020} for the
formulation adapted to the present setting. Since $f$ is gradient-independent, the nonlinear operator
$F(v)(t,x)=f(v(t,x))$ acts pointwise, which is exactly what the MLP recursion exploits.

For every fixed $\omega\in\Omega$, all Brownian increments and random times in~\eqref{eq:mlp}
are deterministic, so that $x\mapsto U^\vartheta_{n,M}(0,x)(\omega)$ is obtained from $g$ and $f$
through finitely many affine shifts, compositions, scalar multiplications, and additions; this is
used in \Cref{sec:repr} to represent one realization of the MLP approximation by a ResNet.

\subsection{Two quantitative inputs}

\Cref{lem:apriori,lem:pde-stab,cor:mlp-stab} specialize
\cite[Lemmas~2.2--2.3 and Corollary~2.4]{HJKN2020} to \Cref{set:mlp}; the random-time MLP error
factor used in \Cref{cor:mlp-stab} comes from \cite[Theorem~3.5]{HJKNW2018}. We include the proofs
in \Cref{app:mlp-proofs} to record the constants needed below. \Cref{lem:mlp-cost} is the
elementary balancing estimate for this error factor.

\begin{lemma}[A priori moment bound, {\cite[Lemma~2.2]{HJKN2020}}]\label{lem:apriori}
In \Cref{set:mlp}, the solutions $v_1,v_2$ of~\eqref{eq:sfp} satisfy, for all $i \in \{1, 2\}$, $x\in\R^d$,
\[
  \sup_{t\in[0,T]}\big(\E\big[\abs{v_i(t,x+\mathbf{W}_t)}^q\big]\big)^{1/q}
  \le e^{LT}(T+1)B\Big[\sup_{t\in[0,T]}\big(\E\big[(1+\norm{x+\mathbf{W}_t})^{pq}\big]\big)^{1/q}\Big].
\]
\end{lemma}

\begin{lemma}[PDE stability, {\cite[Lemma~2.3]{HJKN2020}}]\label{lem:pde-stab}
In \Cref{set:mlp}, the solutions $v_1,v_2$ of~\eqref{eq:sfp} satisfy, for all $t\in[0,T]$, $x\in\R^d$,
\[
\begin{aligned}
  &\E\big[\abs{v_1(t,x+\mathbf{W}_t)-v_2(t,x+\mathbf{W}_t)}\big]\\
  &\quad\le \Delta\,\big(e^{LT}(T+1)\big)^{q+1}(B^q+1)
      \Big(1+\norm{x}+\big(\E[\norm{\mathbf{W}_T}^{pq}]\big)^{1/(pq)}\Big)^{pq}.
\end{aligned}
\]
\end{lemma}

\begin{corollary}[MLP stability bound, {\cite[Corollary~2.4]{HJKN2020}}]\label{cor:mlp-stab}
In \Cref{set:mlp}, for $x\in\R^d$ and $N,M\in\N$, the surrogate-data estimator
$U^{2,0}_{N,M}$ approximates the true solution $v_1$ with
\begin{equation}\label{eq:mlp-stab}
\begin{aligned}
  \big(\E\big[\abs{U^{2,0}_{N,M}(0,x)-v_1(0,x)}^2\big]\big)^{1/2}
  &\le \big(e^{LT}(T+1)\big)^{q+1}(B^q+1)
      \Big(\Delta+\frac{e^{M/2}(1+2LT)^N}{M^{N/2}}\Big)\\[2pt]
  &\quad\times
      \Big(1+\norm{x}+\big(\E[\norm{\mathbf{W}_T}^{pq}]\big)^{\frac1{pq}}\Big)^{pq}.
\end{aligned}
\end{equation}
\end{corollary}

For $M=N$, the following elementary estimate balances the MLP error factor with a recursive
representation factor $(AN)^N$.

\begin{lemma}[MLP balancing estimate]\label{lem:mlp-cost}
	Let $L,T\in(0,\infty)$, $A\in[1,\infty)$, and $\xi\in(0,\infty)$. Then there exists
	$C_{\xi,A}\in(0,\infty)$, depending only on $L,T,\xi,A$, such that, for every
	$\varepsilon\in(0,1]$, there exists $N\in\N\cap[2,\infty)$ satisfying
	\begin{equation}\label{eq:mlp-target}
		\frac{e^{N/2}(1+2LT)^N}{N^{N/2}}\le\varepsilon
	\end{equation}
	and
	\begin{equation}\label{eq:mlp-balance-cost}
		(AN)^N\le C_{\xi,A}\varepsilon^{-(2+\xi)}.
	\end{equation}
\end{lemma}

\section{ResNet representation of MLP estimators}\label{sec:repr}

In this section we establish that every fixed realization of the MLP approximation can be
represented by a ResNet with a polynomially bounded number of parameters. We first formulate the
data approximation hypotheses in \Cref{set:nn}. The one-dimensional approximation result in
\Cref{lem:1d-approx} yields, in \Cref{cor:trunc-surrogate,lem:ridge-surrogate}, explicit families
which satisfy these hypotheses. Thereafter, \Cref{lem:atoms} constructs the elementary
residual updates, \Cref{prop:repr} represents one realization of the MLP approximation, and
\Cref{lem:prec} estimates the number of parameters of the resulting ResNet. The following setting
is the residual-network analogue of \cite[Setting~4.1]{BaggenstosSalimova2023} and
\cite[Theorem~1.1]{HJKN2020}.

\begin{setting}[Network-representable data]\label{set:nn}
	Let $a\in C(\R,\R)$, let $L,B\in(0,\infty)$ and
	$p\in[1,\infty)$, let $f\in C(\R,\R)$ satisfy
	$\Lip(f)\le L$, and, for every $d\in\N$, let
	$g_d\in C(\R^d,\R)$. Assume there exist families of FNNs
	$(\mathfrak f_\delta)_{\delta\in(0,1]}\subseteq\bN$ with
	$\mathcal I(\mathfrak f_\delta)=\mathcal O(\mathfrak f_\delta)=1$ and, for each $d\in\N$,
	$(\mathfrak g_{d,\delta})_{\delta\in(0,1]}\subseteq\bN$ with
	$\mathcal I(\mathfrak g_{d,\delta})=d$ and $\mathcal O(\mathfrak g_{d,\delta})=1$,
	and constants $\kappa,C_g,L_f\in(0,\infty)$ such that for all $d\in\N$, $\delta\in(0,1]$, $v,w\in\R$, $x\in\R^d$:
\begin{gather}
  \abs{f(v)-(\cR_a\mathfrak{f}_\delta)(v)}\le \delta, \quad
  \abs{g_d(x)-(\cR_a\mathfrak{g}_{d,\delta})(x)}
   \le C_g\,\delta\, d^\kappa(1+\norm{x}^\kappa),\label{eq:nn-acc}\\
  \begin{aligned}
  \abs{(\cR_a\mathfrak f_\delta)(w)-(\cR_a\mathfrak f_\delta)(v)}&\le L_f\abs{w-v},\qquad
   \abs{(\cR_a\mathfrak f_\delta)(0)}\le B,\\
  \abs{(\cR_a\mathfrak g_{d,\delta})(x)}&\le B(1+\norm x)^p,
  \end{aligned}\label{eq:nn-surr}\\
  \cP(\mathfrak{f}_\delta)+\cP(\mathfrak{g}_{d,\delta})\le C_g\,d^\kappa\,\delta^{-\kappa}.
 \label{eq:nn-cost}
\end{gather}
\end{setting}

\begin{remark}\label{rem:setNN-content}
The polynomial bound in~\eqref{eq:nn-cost} expresses that the data $f$ and $g_d$ do not
themselves carry the curse of dimensionality. This hypothesis is analogous to the corresponding
assumptions in~\cite{HJKN2020,BaggenstosSalimova2023}. In contrast to the feedforward arguments
in~\cite{HJKN2020,JentzenSalimovaWelti2021}, however, we do not require that the activation
function $a$ admit an exact FNN representation of the identity. Indeed, the residual updates in
\Cref{prop:repr} are combined by the concatenation rule in \Cref{lem:resnet-comp}, and no identity
subnetwork is inserted between consecutive updates.

When \Cref{set:mlp,set:nn} are combined, the bounds in~\eqref{eq:nn-surr} allow the true and
surrogate data to be controlled by common constants, uniformly in the approximation accuracy
$\delta$. Without a uniform Lipschitz bound for
$\cR_a(\mathfrak f_\delta)$, the constants in the stability estimate of \Cref{cor:mlp-stab} would
depend on $\delta$, and the polynomial complexity estimate would in general be lost. We therefore
replace $L$ by $\max\{L,L_f\}$ and enlarge $B$, if necessary, so that $B$ dominates both the
surrogate value bound $B_f$ in \Cref{cor:trunc-surrogate} and the uniform bound for the
$g_d$-surrogates in \Cref{lem:ridge-surrogate}. Henceforth, $L$ and $B$ denote these enlarged
constants.

\end{remark}

We next verify the hypotheses of \Cref{set:nn} for a concrete class of initial conditions. The
class consists of finite sums of ridge functions of the form
\[
  g_d(x)=\sum_{j=1}^{J}\phi_j\big(\langle\alpha_j^d,x\rangle\big).
\]
Such functions arise, for example, as bounded truncations of European basket option payoffs,
 as single-neuron initial conditions for Allen--Cahn-type equations, and as initial conditions
for Fisher--Kolmogorov-type fronts propagating in a fixed direction.

We first establish the one-dimensional approximation result; its final assertion, proved by a
variation-adapted staircase, serves the ridge profiles of \Cref{lem:ridge-surrogate}. For
$h\colon\R\to\R$ and an interval $I\subseteq\R$ we write $\mathrm{TV}(h;I)$ for the total variation
of $h$ on $I$, that is, the supremum of $\sum_k\abs{h(t_k)-h(t_{k-1})}$ over all finite increasing
sequences $(t_k)$ in $I$, and $\mathrm{TV}(h):=\mathrm{TV}(h;\R)$.

\begin{lemma}[One-dimensional sigmoidal approximation]\label{lem:1d-approx}
	Let $a\in C^1(\R,\R)$ be nondecreasing and assume that the finite limits
	\begin{equation}\label{eq:a-sigmoidal}
		A_-:=\lim_{r\to-\infty}a(r)\in\R\qquad\text{and}\qquad A_+:=\lim_{r\to+\infty}a(r)\in\R
	\end{equation}
	exist and satisfy $A_-<A_+$.
	Let $\widetilde a:=(a-A_-)/(A_+-A_-)$, and assume that
	\begin{equation}\label{eq:a-admissible}
		\mathrm{TV}(\widetilde a')<\infty,\qquad
		\widetilde\Lambda_a
		:=\int_\R\abs{\widetilde a(r)-\mathbf 1_{(0,\infty)}(r)}\,dr<\infty.
	\end{equation}
	We call such an $a$ an \emph{admissible sigmoidal activation}.
	Then there exists $C_a\in[1,\infty)$, depending only on $a$, such that for every Lipschitz
	function $\phi\colon\R\to\R$, every $R\in[1,\infty)$, and every $n\in\N$, there exists an FNN
	$\phi_{n,R}\in\bN$ satisfying
	\begin{equation}\label{eq:1d-architecture}
		\cD(\phi_{n,R})=(1,n,1),\qquad \cP(\phi_{n,R})=3n+1,
	\end{equation}
	\begin{equation}\label{eq:1d-approx-in}
		\sup_{v\in[-R,R]}\abs{\phi(v)-(\cR_a\phi_{n,R})(v)}
		\le C_a\Lip(\phi)\frac{R}{n},
	\end{equation}
	and
	\begin{equation}\label{eq:1d-lip}
		\Lip\big(\cR_a\phi_{n,R}\big)\le C_a\Lip(\phi).
	\end{equation}
	If, in addition, $\phi$ is constant on each of
	$(-\infty,-R]$ and $[R,\infty)$, then
	\begin{equation}\label{eq:1d-approx-global}
		\sup_{v\in\R}\abs{\phi(v)-(\cR_a\phi_{n,R})(v)}
		\le C_a\Lip(\phi)\frac{R}{n}.
	\end{equation}
	Finally, if $\phi$ has finite total variation, then, for every $\delta\in(0,1]$, there exist
	$m\in\N$ and an FNN $\phi_\delta\in\bN$ with
	\begin{equation}\label{eq:1d-bv}
		\cD(\phi_\delta)=(1,m,1),\qquad
		\sup_{v\in\R}\abs{\phi(v)-(\cR_a\phi_\delta)(v)}\le\delta,\qquad
		m\le\max\Big\{1,\Big\lceil\frac{3\,\mathrm{TV}(\phi)}{\delta}\Big\rceil-1\Big\}.
	\end{equation}
\end{lemma}

The proof of \Cref{lem:1d-approx} is provided in \Cref{app:1d-proof}.

\begin{remark}\label{rem:smooth-activ}
The logistic function, $\tanh$, the error function, and the algebraic sigmoid
$r\mapsto r/\sqrt{1+r^2}$ are admissible sigmoidal activations: in each case $\widetilde a'$ is
unimodal, hence of bounded variation, and $\widetilde a$ approaches its limits at an integrable
rate. None of them represents the identity through the two-unit construction
of~\cite{HJKN2020,JentzenSalimovaWelti2021}; conversely, $\mathrm{ReLU}$ does not
satisfy~\eqref{eq:a-sigmoidal} but represents the identity through
$r=\mathrm{ReLU}(r)-\mathrm{ReLU}(-r)$ and is therefore covered by~\cite{HJKN2020}.
\end{remark}

\begin{corollary}[Uniformly Lipschitz surrogates for globally Lipschitz truncations]\label{cor:trunc-surrogate}
	Let $a$ be an admissible sigmoidal activation in the sense of \Cref{lem:1d-approx}, let
	$L,M_f\in(0,\infty)$, and let $f\in C(\R,\R)$ satisfy $\Lip(f)\le L$ and be constant on each
	of $(-\infty,-M_f]$ and $[M_f,\infty)$---a \emph{globally Lipschitz truncation}. Then there
	exist a family $(\mathfrak f_\delta)_{\delta\in(0,1]}\subseteq\bN$ and constants
	$L_f,B_f,C_f\in(0,\infty)$, where $L_f$ and $C_f$ depend only on $a,L,M_f$, whereas $B_f$
	depends only on $\abs{f(0)}$, such that
	for every $\delta\in(0,1]$,
	\begin{equation}\label{eq:trunc-surrogate}
		\sup_{v\in\R}\abs{f(v)-(\cR_a\mathfrak f_\delta)(v)}\le\delta,\quad
		\Lip(\cR_a\mathfrak f_\delta)\le L_f,\quad
		\abs{(\cR_a\mathfrak f_\delta)(0)}\le B_f,\quad
		\cP(\mathfrak f_\delta)\le C_f\,\delta^{-1}.
	\end{equation}
	Moreover, $\mathfrak f_\delta$ may be chosen with architecture $(1,n_f,1)$ and
	$n_f\le C_f\delta^{-1}$. In particular $f$ satisfies the $f$-part
	of~\eqref{eq:nn-acc}--\eqref{eq:nn-surr} with uniform Lipschitz constant $L_f$, and the $f$-part
	of~\eqref{eq:nn-cost} with exponent $\kappa=1$.
\end{corollary}

\begin{proof}[Proof of \Cref{cor:trunc-surrogate}]
	Set $R:=M_f+1$, and let $C_a$ be the constant in \Cref{lem:1d-approx}. For
	$\delta\in(0,1]$, choose
	\[
	n:=\max\{1,\lceil C_aLR/\delta\rceil\}
	\]
	and apply \Cref{lem:1d-approx} to $f$. Since $f$ is constant on $(-\infty,-R]$ and on
	$[R,\infty)$, \eqref{eq:1d-approx-global} gives
	\[
	\sup_{v\in\R}\abs{f(v)-(\cR_a\mathfrak f_\delta)(v)}
	\le \frac{C_aLR}{n}\le\delta,
	\qquad
	\Lip(\cR_a\mathfrak f_\delta)\le C_aL.
	\]
	Consequently,
	\[
	\abs{(\cR_a\mathfrak f_\delta)(0)}\le\abs{f(0)}+1.
	\]
	Moreover, \eqref{eq:1d-architecture} and $\delta\le1$ yield
	\[
	\cD(\mathfrak f_\delta)=(1,n,1),\qquad
	n\le(1+C_aLR)\delta^{-1},\qquad
	\cP(\mathfrak f_\delta)=3n+1\le(4+3C_aLR)\delta^{-1}.
	\]
	The claimed constants may therefore be chosen as $L_f:=C_aL$, $B_f:=1+\abs{f(0)}$, and
	$C_f:=4+3C_aLR$. These estimates establish the $f$-parts of~\eqref{eq:nn-acc}--\eqref{eq:nn-cost}.
\end{proof}

\begin{lemma}[Ridge-sum approximation]\label{lem:ridge-surrogate}
	Let $a$ be an admissible sigmoidal activation in the sense of \Cref{lem:1d-approx}, let
	$J\in\N$, $L_\phi\in(0,\infty)$, and $M_\phi,T_\phi\in[0,\infty)$, and let
	$\phi_1,\dots,\phi_J\in C(\R,\R)$ satisfy
	\begin{equation}\label{eq:ridge-form}
		\Lip(\phi_j)\le L_\phi,\qquad
		\abs{\phi_j(0)}\le M_\phi,\qquad
		\mathrm{TV}(\phi_j)\le T_\phi
	\end{equation}
	for all $j\in\{1,\dots,J\}$. For every $d\in\N$, let $\alpha^d_1,\dots,\alpha^d_J\in\R^d$ and
	define
	\begin{equation}\label{eq:ridge-data}
		g_d(x):=\sum_{j=1}^{J}\phi_j\big(\langle\alpha^d_j,x\rangle\big)
		\qquad (x\in\R^d).
	\end{equation}
	Then there exists $C\in(0,\infty)$, depending only on $J$ and $T_\phi$, such that for all
	$d\in\N$ and $\delta\in(0,1]$ there is an FNN $\mathfrak g_{d,\delta}\in\bN$ of architecture
	$(d,n_{d,\delta},1)$ with $n_{d,\delta}\le C\delta^{-1}$ satisfying
	\begin{equation}\label{eq:ridge-bounds}
		\sup_{x\in\R^d}\abs{g_d(x)-(\cR_a\mathfrak g_{d,\delta})(x)}\le\delta,
		\qquad
		\cP(\mathfrak g_{d,\delta})\le C\,d\,\delta^{-1},
	\end{equation}
	and
	\begin{equation}\label{eq:ridge-sup}
		\sup_{x\in\R^d}\abs{(\cR_a\mathfrak g_{d,\delta})(x)}\le J(M_\phi+T_\phi)+1.
	\end{equation}
\end{lemma}

\begin{proof}[Proof of \Cref{lem:ridge-surrogate}]
	Fix $d\in\N$ and $\delta\in(0,1]$, and set $\delta_0:=\delta/J$. By the final assertion of
	\Cref{lem:1d-approx}, for every $j\in\{1,\dots,J\}$ there exists an FNN $\phi_{j,\delta_0}\in\bN$
	of architecture $(1,m_j,1)$ with
	\[
	\sup_{v\in\R}\abs{\phi_j(v)-(\cR_a\phi_{j,\delta_0})(v)}\le\delta_0,
	\qquad
	m_j\le\max\Big\{1,\Big\lceil\frac{3JT_\phi}{\delta}\Big\rceil-1\Big\}.
	\]
	By adding zero rows to $W^{(1)}_j$, zero entries to $B^{(1)}_j$, and matching zero columns to
	$W^{(2)}_j$, which changes neither the realization nor the order of the parameter count, we may
	assume that all these networks have the common architecture $(1,n,1)$ with $n:=\max\{1,\lceil3JT_\phi/\delta\rceil\}$. Write
	$\phi_{j,\delta_0}=\big((W^{(1)}_j,B^{(1)}_j),(W^{(2)}_j,B^{(2)}_j)\big)$ with
	$W^{(1)}_j\in\R^{n\times1}$, and define an FNN of architecture $(d,Jn,1)$ by stacking the lifted
	first layers and concatenating the output layers:
	\[
	\mathfrak g_{d,\delta}:=\big((W^{(1)},B^{(1)}),(W^{(2)},B^{(2)})\big),
	\qquad
	W^{(1)}:=\begin{pmatrix}W^{(1)}_1(\alpha^d_1)^{\mathsf T}\\ \vdots\\ W^{(1)}_J(\alpha^d_J)^{\mathsf T}\end{pmatrix},
	\qquad
	B^{(1)}:=\begin{pmatrix}B^{(1)}_1\\ \vdots\\ B^{(1)}_J\end{pmatrix},
	\]
	\[
	W^{(2)}:=\big(W^{(2)}_1\ \cdots\ W^{(2)}_J\big),
	\qquad
	B^{(2)}:=\sum_{j=1}^{J}B^{(2)}_j.
	\]
	Since the $j$th block of hidden units receives the input
	$W^{(1)}_j\langle\alpha^d_j,x\rangle+B^{(1)}_j$, \Cref{def:fnn} gives, for every $x\in\R^d$,
	\begin{equation}\label{eq:gd-sum}
		(\cR_a\mathfrak g_{d,\delta})(x)=\sum_{j=1}^{J}(\cR_a\phi_{j,\delta_0})(\langle\alpha^d_j,x\rangle),
	\end{equation}
	and therefore
	\begin{equation}\label{eq:gd-err-final}
		\sup_{x\in\R^d}\abs{g_d(x)-(\cR_a\mathfrak g_{d,\delta})(x)}
		\le\sum_{j=1}^{J}\sup_{v\in\R}\abs{\phi_j(v)-(\cR_a\phi_{j,\delta_0})(v)}
		\le J\delta_0=\delta.
	\end{equation}
	Moreover, $n_{d,\delta}=Jn$, and since $\delta\le1$ gives $n\le(3JT_\phi+1)\delta^{-1}$,
	\begin{equation}\label{eq:gd-cost-final}
		\cP(\mathfrak g_{d,\delta})=Jn(d+1)+(Jn+1)=Jn(d+2)+1\le4Jnd
		\le4J\big(3JT_\phi+1\big)\,d\,\delta^{-1},
	\end{equation}
	which proves~\eqref{eq:ridge-bounds} with $C:=4J(3JT_\phi+1)$. Finally,
	$\sup_{v\in\R}\abs{\phi_j(v)}\le\abs{\phi_j(0)}+\mathrm{TV}(\phi_j)\le M_\phi+T_\phi$, so
	that~\eqref{eq:ridge-data} and~\eqref{eq:gd-err-final} give~\eqref{eq:ridge-sup}.
\end{proof}

Together with \Cref{cor:trunc-surrogate}, \Cref{lem:ridge-surrogate} verifies \Cref{set:nn} with
$p=\kappa=1$ for globally Lipschitz truncations and the ridge-sum initial conditions above:
the $f$-parts of~\eqref{eq:nn-acc}--\eqref{eq:nn-cost} are~\eqref{eq:trunc-surrogate}, and the
$g$-parts follow from~\eqref{eq:ridge-bounds}--\eqref{eq:ridge-sup} since $1+\norm x^\kappa\ge1$
and $d\ge1$, after enlarging $C_g$ to absorb $C_f$ and $C$.

\begin{remark}[Breadth and limitations of the ridge-sum class]\label{rem:applications}
\Cref{lem:ridge-surrogate} covers shallow-network initial conditions
$g_d(x)=\sum_{j=1}^{J}c_j\,b(\langle\alpha^d_j,x\rangle+\beta_j)$ with a Lipschitz profile $b$ of
finite total variation and $J$ independent of $d$, including single-index models ($J=1$) and the
Allen--Cahn-type datum of \Cref{rem:allen-cahn-example}, with no condition on the direction
vectors. It does not cover a general radial function $g_d(x)=\psi(\norm{x})$ or a generic Lipschitz
function on $\R^d$; the lower bounds of~\cite{GrohsIbragimovJentzenKoppensteiner2023,GrohsVoigtlaender2024}
discussed in \Cref{sec:rigorous} show that such restrictions cannot be removed in general, and
verifying \Cref{set:nn} for further structured classes (tensor products, hierarchical compositions,
Barron-type representations) is left to future research. An unbounded basket payoff
$(K-\langle w^d,x\rangle)_+$ is not of finite total variation, but a bounded Lipschitz truncation
of it is covered, with no condition on $w^d$, provided the truncated profiles have
dimension-uniform Lipschitz, value, and total-variation bounds; comparing with the PDE for the
original payoff then requires a separate Gaussian tail estimate, since agreement of the initial
conditions on the test cube alone is not sufficient for the nonlocal heat flow.
\end{remark}

\subsection{ResNet representation of an MLP sample}

The representation uses the last coordinate of the state as a scalar accumulator. Each
accumulator update adds one summand of the MLP formula to this coordinate; depending on the
summand, an update consists of a single residual block or a short concatenation of such blocks,
and each shortcut transmits the
spatial variable. The resulting blocks are combined by concatenation. Consequently, the
construction requires no common-width padding.

\begin{lemma}[Atomic accumulator blocks]\label{lem:atoms}
Let $a\in C(\R,\R)$.
\begin{enumerate}[label=\textup{(\alph*)},leftmargin=2.2em,itemsep=3pt,topsep=2pt]
\item Let $r,D\in\N$ with $r\le D$, let $\varphi\in\bN$ have input dimension $r$ and scalar output,
let $P\in\R^{r\times D}$, $b\in\R^r$, $e\in\R^D$, and $h\in\R$. There exists a one-block ResNet
$\mathcal A_{\varphi,P,b,e,h}\in\bbN_1$ of input/output dimension $D$ such that
\begin{equation}\label{eq:atomic-add}
   \Real_a(\mathcal A_{\varphi,P,b,e,h})(z)
   =z+h e\,(\cR_a\varphi)(Pz+b),
   \qquad z\in\R^D,
\end{equation}
and
\begin{equation}\label{eq:atomic-add-cost}
   \Param(\mathcal A_{\varphi,P,b,e,h})
   \le C\big(D\,\cP(\varphi)+D^2\big)
\end{equation}
for a universal $C\in(0,\infty)$. If, in addition,
$\cD(\varphi)=(D-1,m,1)$ and $P=(\id_{D-1}\mid0)$, then the sharper estimate
\begin{equation}\label{eq:atomic-add-cost-sharp}
   \Param(\mathcal A_{\varphi,P,b,e,h})
   \le C\big(\cP(\varphi)+D^2\big)
\end{equation}
holds.
\item For every $d\in\N$ and $b\in\R^d$ there exist one-block ResNets $\mathcal I_b$ and
$\mathcal O_b$ whose input-output dimensions are, respectively, $(d+1,d+2)$ and $(d+2,d+1)$,
such that
\begin{equation}\label{eq:init-restore}
  \Real_a(\mathcal I_b)(x,s)=(x+b,0,s),
  \qquad
  \Real_a(\mathcal O_b)(y,w,s)=(y-b,s),
\end{equation}
and
$\Param(\mathcal I_b)+\Param(\mathcal O_b)\le C d^2$.
\end{enumerate}
\end{lemma}

\begin{proof}[Proof of \Cref{lem:atoms}]
For part~(a), write $\cD(\varphi)=(r,m_1,\dots,m_{L-1},1)$. If $L=1$, then
$\cR_a\varphi(v)=Wv+B$ is affine. The map
\[
   z\longmapsto he\big(W(Pz+b)+B\big)
\]
is the realization of an affine FNN $\widehat\varphi\colon\R^D\to\R^D$ with
$D(D+1)$ parameters. If $L\ge2$, replace the first affine pair $(W_1,B_1)$ by
$(W_1P,W_1b+B_1)$, replace the last affine pair $(W_L,B_L)$ by
$(heW_L,heB_L)$, and leave all intermediate layers unchanged. The resulting FNN
$\widehat\varphi\colon\R^D\to\R^D$ has
\[
   (\cR_a\widehat\varphi)(z)=he\,(\cR_a\varphi)(Pz+b).
\]
Because $1\le r\le D$, the new first layer has at most $D$ times as many parameters as the old first
layer, the new last layer has exactly $D$ times as many parameters as the old last layer, and the
intermediate layers are unchanged. Hence
$\cP(\widehat\varphi)\le D\cP(\varphi)$ for $L\ge2$, while the affine case is bounded by $D(D+1)$.
Let the shortcut of the single residual block be $\id_D$. This proves
\eqref{eq:atomic-add} and, after adding the $D^2$ shortcut parameters,
\eqref{eq:atomic-add-cost}. If $\cD(\varphi)=(D-1,m,1)$ and
$P=(\id_{D-1}\mid0)$, then $\widehat\varphi$ has architecture $(D,m,D)$ and
\[
   \cP(\widehat\varphi)=m(D+1)+D(m+1)
   \le C\big(\cP(\varphi)+D^2\big).
\]
Adding the shortcut parameters gives~\eqref{eq:atomic-add-cost-sharp}.

For part~(b), use shortcuts
\[
  (x,s)\longmapsto(x,0,s),
  \qquad
  (y,w,s)\longmapsto(y,s),
\]
and constant residual FNNs with values $(b,0,0)$ and $(-b,0)$, respectively. The realizations are
\eqref{eq:init-restore}, and the two dense affine blocks and shortcuts contain at most $Cd^2$
parameters.
\end{proof}

\begin{proposition}[MLP-sample accumulator representation]\label{prop:repr}
Assume \Cref{set:nn} and use the probability space and random variables of \Cref{set:mlp}. Fix
$d\in\N$, $\delta\in(0,1]$, $M,N\in\N$, and a realization $\omega\in\Omega$. For
$\vartheta\in\mathfrak T$ and $n\in\N_0$, let $U^{\vartheta,\delta}_{n,M}$ denote the MLP
estimator~\eqref{eq:mlp} obtained by taking
$g_2=\cR_a\mathfrak g_{d,\delta}$ and $f_2=\cR_a\mathfrak f_\delta$; thus $\delta$ here is the
surrogate accuracy and is unrelated to the perturbation parameter $\Delta$ of \Cref{set:mlp}.
Then there exists a ResNet
$\Psi^\omega_{N,M,\delta}\in\bbN$ such that
\begin{equation}\label{eq:repr-real}
  (\Real_a\Psi^\omega_{N,M,\delta})(x)
  =U^{0,\delta}_{N,M}(0,x)(\omega),
  \qquad x\in\R^d.
\end{equation}
More precisely, for every occurrence of a level-$n$ subestimator in the recursion, specified by
its multi-index $\vartheta\in\mathfrak T$, initial time $\tau\in[0,T]$, and the fixed realization
$\omega$, there exists an accumulator ResNet
$\mathcal A^{(n,\vartheta,\tau,\omega)}\in\bbN$ with
$\mathcal I(\mathcal A^{(n,\vartheta,\tau,\omega)})
=\mathcal O(\mathcal A^{(n,\vartheta,\tau,\omega)})=d+1$ and satisfying
\begin{equation}\label{eq:accumulator-real}
   \Real_a\!\left(\mathcal A^{(n,\vartheta,\tau,\omega)}\right)(x,s)
   =\bigl(x,s+U^{\vartheta,\delta}_{n,M}(\tau,x)(\omega)\bigr).
\end{equation}
\end{proposition}

\begin{proof}[Proof of \Cref{prop:repr}]
Fix $\omega$ and suppress it from the notation. Every randomized time and Brownian increment in
\eqref{eq:mlp} is then fixed. We prove~\eqref{eq:accumulator-real} by induction on $n$, uniformly
over all occurrences $(\vartheta,\tau)$; when no confusion is possible, we suppress these indices
from the accumulator notation.

For $n=0$, the estimator is zero, and the one-block ResNet on $\R^{d+1}$ with shortcut
$\id_{d+1}$ and zero residual FNN satisfies~\eqref{eq:accumulator-real}.

Let $n\ge1$ and assume the assertion for all lower levels. Consider~\eqref{eq:mlp} at a fixed time
$\tau$. Its first sum contains $M^n$ summands of the form
\[
   M^{-n}(\cR_a\mathfrak g_{d,\delta})(x+\beta).
\]
For each such summand, \Cref{lem:atoms}(a), applied with $D=d+1$, $r=d$,
$P=(\id_d\mid0)$, $b=\beta$, $e=e_{d+1}$, and $h=M^{-n}$, provides a residual block which leaves
$x$ unchanged and adds the displayed quantity to the accumulator $s$.

Next, consider a summand of the form
\begin{equation}\label{eq:f-leaf-generic}
   h\,(\cR_a\mathfrak f_\delta)\big(U^{\vartheta',\delta}_{l,M}(\tau',x+b)\big),
\end{equation}
where $l<n$, $\vartheta'$ is the multi-index of this inner occurrence, and $h$ is the fixed
coefficient, positive or negative, appearing in~\eqref{eq:mlp}.
The induction hypothesis establishes the existence of an accumulator ResNet on $\R^{d+1}$ carrying
the multi-index and time of this particular inner estimator. Denote it temporarily by
$\mathcal A^{(l)}$. The ResNet $\mathcal I_b$ maps $(x,s)$ to
$(y,w,s)=(x+b,0,s)$. The one-coordinate lift of $\mathcal A^{(l)}$ furnished by
\Cref{lem:resnet-lift} maps this state to
$(y,U^{\vartheta',\delta}_{l,M}(\tau',y),s)$. Thereafter,
\Cref{lem:atoms}(a), applied with $D=d+2$, $r=1$, $b=0$, with $P$ selecting the coordinate $w$, and with
$e=e_{d+2}$, adds $h(\cR_a\mathfrak f_\delta)(w)$ to the final coordinate. The ResNet
$\mathcal O_b$ then maps the state to
$(x,s+h(\cR_a\mathfrak f_\delta)(U^{\vartheta',\delta}_{l,M}(\tau',x+b)))$. Consequently,
\Cref{lem:resnet-comp} shows that the concatenation of these four ResNets realizes the update associated with
\eqref{eq:f-leaf-generic}. The correction terms involving
$U^{\vartheta'',\delta}_{l-1,M}$, which occur only for $l\in\{1,\dots,n-1\}$, are handled
identically, with the multi-index and time belonging to that occurrence.

All the resulting updates have input and output dimension $d+1$. Concatenating them therefore
adds every summand of~\eqref{eq:mlp} to the same accumulator and leaves the spatial variable unchanged.
This proves~\eqref{eq:accumulator-real} at level $n$.

Finally, concatenating the linear one-block ResNet $x\mapsto(x,0)$, the ResNet
$\mathcal A^{(N,0,0,\omega)}$, and the linear one-block readout $(x,s)\mapsto s$ yields the ResNet
$\Psi^\omega_{N,M,\delta}$ which satisfies~\eqref{eq:repr-real}.
\end{proof}

\subsection{The parameter recursion}

For arbitrary surrogate depths define
\begin{equation}\label{eq:Bdd-hat}
   \widehat B_{d,\delta}
   :=d\big(\cP(\mathfrak g_{d,\delta})+\cP(\mathfrak f_\delta)\big)+c_0d^2,
\end{equation}
where $c_0\in(0,\infty)$ is universal. If $\mathfrak g_{d,\delta}$ has one hidden layer, as in
\Cref{lem:ridge-surrogate}, set instead
\begin{equation}\label{eq:Bdd}
   B_{d,\delta}
   :=\cP(\mathfrak g_{d,\delta})+d\,\cP(\mathfrak f_\delta)+c_0d^2.
\end{equation}
The sharper coefficient of the $g$-surrogate follows from~\eqref{eq:atomic-add-cost-sharp}.

\begin{lemma}[Parameter recursion]\label{lem:prec}
Let $d\in\N$ and $\delta\in(0,1]$. There exist universal constants $c_1\in(0,\infty)$ and
$K\in[1,\infty)$ and, for every $M\in\N\cap[2,\infty)$ and $n\in\N_0$, a deterministic bound
$p_{n,M}\in[0,\infty)$ on the number of parameters of every level-$n$ accumulator constructed with
branching parameter $M$ in the proof of \Cref{prop:repr}, uniformly over all multi-indices, initial
times, Brownian increments, and random times (these values change only the weights and biases,
not the architecture), such that, for every $n\in\N$,
\begin{equation}\label{eq:prec}
\begin{aligned}
  p_{n,M}&\le c_1\Big[M^n\mathcal B_{d,\delta}
      +\sum_{l=0}^{n-1}M^{n-l}\big(p_{l,M}+\mathcal B_{d,\delta}\big)\Big],\\
  p_{0,M}&\le c_1\mathcal B_{d,\delta}.
\end{aligned}
\end{equation}
Here $\mathcal B_{d,\delta}=\widehat B_{d,\delta}$ in general. For the one-hidden-layer
$g$-surrogates of \Cref{lem:ridge-surrogate}, one may take
$\mathcal B_{d,\delta}=B_{d,\delta}$. Consequently,
\begin{equation}\label{eq:prec-closed}
   p_{N,M}\le c_1\,\mathcal B_{d,\delta}\,(KM)^N,
   \qquad M\in\N\cap[2,\infty),\ N\in\N,
\end{equation}
and the scalar-output ResNet of \Cref{prop:repr} satisfies the same bound after enlarging $c_1$.
\end{lemma}

\begin{proof}[Proof of \Cref{lem:prec}]
Fix $M\in\N\cap[2,\infty)$. By \Cref{lem:atoms}(a), an update associated with a summand containing
$g$ has at most
$C\mathcal B_{d,\delta}$ parameters. For an update associated with a summand containing
$f\circ U_l$, the initialization and restoration blocks together have at most $Cd^2$ parameters,
the lifted level-$l$ accumulator has at most $4p_{l,M}$ parameters by \Cref{lem:resnet-lift}, and the
reaction block has at most $C(d\cP(\mathfrak f_\delta)+d^2)$ parameters. Hence, every such update
has at most $C(p_{l,M}+\mathcal B_{d,\delta})$ parameters.

There are $M^n$ summands involving $g$ and $M^{n-l}$ summands of the first nonlinear type at level
$l$. Reindexing the correction terms with $j=l-1$ and using $M^{n-j-1}\le M^{n-j}$ shows that
they contribute at most $\sum_{j=0}^{n-1}M^{n-j}(p_{j,M}+\mathcal B_{d,\delta})$. Because concatenation adds parameter counts exactly, \Cref{lem:resnet-comp}(ii) gives
\eqref{eq:prec}. The identity accumulator in the base case has at most $Cd^2$ parameters.

Let $q_{n,M}:=p_{n,M}/(M^n\mathcal B_{d,\delta})$. Dividing~\eqref{eq:prec} by
$M^n\mathcal B_{d,\delta}$, noting that $M^{n-l}(p_{l,M}+\mathcal B_{d,\delta})$ equals
$(q_{l,M}+M^{-l})\,M^n\mathcal B_{d,\delta}$, and using
$\sum_{l\ge0}M^{-l}\le2$ yields
\[
   q_{n,M}\le C\Big(1+\sum_{l=0}^{n-1}q_{l,M}\Big).
\]
If $b_{n,M}:=1+\sum_{l=0}^n q_{l,M}$, then the base case of~\eqref{eq:prec} gives
$b_{0,M}=1+q_{0,M}\le1+c_1$, and $b_{n,M}\le(1+C)b_{n-1,M}$ for $n\ge1$. Hence
$q_{n,M}\le C\,b_{n-1,M}\le C(1+c_1)(1+C)^{n-1}$, and~\eqref{eq:prec-closed} follows with $K:=1+C$ after changing the
universal constant. The two root blocks $x\mapsto(x,0)$ and $(x,s)\mapsto s$ together have at most
$Cd^2$ parameters, and this quantity is absorbed by $\mathcal B_{d,\delta}$.
\end{proof}

\section{Main theorem}\label{sec:main}

\begin{theorem}[Residual neural networks overcome the curse of dimensionality for semilinear heat equations]\label{thm:main}
Let $L\in(0,\infty)$, let $a\in C(\R,\R)$, and let $f\in C(\R,\R)$ satisfy $\Lip(f)\le L$. For
every $d\in\N$ let $g_d\in C(\R^d,\R)$ and let
$u_d\in C([0,T]\times\R^d,\R)\cap C^{1,2}((0,T]\times\R^d,\R)$ be a solution of at most polynomial
growing solution of~\eqref{eq:pde} in the sense of~\eqref{eq:reg-class}. Assume that there exist
$B\in(0,\infty)$ and $p\in[1,\infty)$ such that
\[
   \abs{f(0)}\le B,\qquad
   \abs{g_d(x)}\le B(1+\norm{x})^p
   \qquad(d\in\N,\ x\in\R^d).
\]
Then the following hold.
\begin{enumerate}[label=\textup{(\roman*)},leftmargin=2.4em,itemsep=3pt,topsep=3pt]
\item Assume that \Cref{set:nn} holds with the constants
$L,B,p$ above and with cost exponent $\kappa$. Then
there exist $\eta\in(0,\infty)$ and a family of ResNets
$(\Psi_{d,\varepsilon})_{d\in\N,\,\varepsilon\in(0,1]}\subseteq\bbN$ such that, for all
$d\in\N$ and $\varepsilon\in(0,1]$,
\begin{equation}\label{eq:main}
\begin{gathered}
  \Real_a(\Psi_{d,\varepsilon})\in C(\R^d,\R),\qquad
  \Param(\Psi_{d,\varepsilon})\le \eta d^\eta\varepsilon^{-\eta},\\[2pt]
  \Big[\int_{[0,1]^d}\abs{u_d(T,x)-(\Real_a\Psi_{d,\varepsilon})(x)}^2\,dx\Big]^{1/2}
  \le\varepsilon.
\end{gathered}
\end{equation}
\item Assume instead that $a$ is an admissible sigmoidal activation in the sense of
\Cref{lem:1d-approx}, that $f$ is a globally Lipschitz truncation, and that $(g_d)_{d\in\N}$ has
the ridge-sum form~\eqref{eq:ridge-form}--\eqref{eq:ridge-data} of \Cref{lem:ridge-surrogate}.
Then the conclusions of part~\textup{(i)} hold and, for every $\xi>0$, the networks may be chosen
so that
\begin{equation}\label{eq:main-sharp}
  \Param(\Psi_{d,\varepsilon})
  \le C_\xi\,
  d^{\,4+\xi}\,
  \varepsilon^{-(3+\xi)}
  \qquad(d\in\N,\ \varepsilon\in(0,1]),
\end{equation}
where $C_\xi$ depends only on $\xi$ and on the dimension-independent constants in the hypotheses,
and not on $d$, $\varepsilon$, or the direction vectors $\alpha^d_j$.
\end{enumerate}
\end{theorem}

\begin{proof}[Proof of \Cref{thm:main}]
Throughout this proof let $d\in\N$, $\varepsilon\in(0,1]$, and $\xi\in(0,\infty)$. For
part~\textup{(i)} it is sufficient to choose, for example, $\xi=1$, whereas for
part~\textup{(ii)} the real number $\xi$ is arbitrary. We denote by $C,C',\dots$ finite
constants which may depend on $T,L,B,p,\kappa,C_g,L_f$, on the constants of
\Cref{lem:prec}, and on $\xi$, but not on $d$, $\varepsilon$, $\rho$, $\delta$, $N$, or $M$. Enlarge $L$ to $\max\{L,L_f\}$ and $B$ if necessary, as explained in
\Cref{rem:setNN-content}, so that the true and surrogate data satisfy the same bounds in
\Cref{set:mlp}, and let $q:=\max\{2,\kappa/p\}$, so that $q\ge2$ and $pq\ge\kappa$; \Cref{set:mlp}
is instantiated with this moment exponent. Let
\[
  \Pi_d(x):=\Big(1+\norm x+\big(\E[\norm{\mathbf W_T}^{pq}]\big)^{1/(pq)}\Big)^{pq}.
\]
Let $C_\Pi\in[1,\infty)$ be an exponent for which there exists $K_\Pi\in[1,\infty)$,
independently of $d$ and $\varepsilon$, such that
\begin{equation}\label{eq:KPi}
  \big(e^{LT}(T+1)\big)^{q+1}(B^q+1)
  \Big(\int_{[0,1]^d}\Pi_d(x)^2\,dx\Big)^{1/2}
  \le K_\Pi d^{C_\Pi}.
\end{equation}
The Gaussian moment estimate
$(\E[\norm{\mathbf W_T}^{pq}])^{1/(pq)}\le C_{p,q,T}\sqrt d$, together with
$\norm{x}\le\sqrt d$ for $x\in[0,1]^d$, shows, in particular, that $C_\Pi=pq/2$ is admissible.

First, we choose the surrogate and MLP approximation accuracies.
Enlarging $C_g$ if necessary, assume $C_g\ge1$. By~\eqref{eq:nn-acc}, $pq\ge\kappa$, and
$1+\norm x^\kappa\le2(1+\norm x)^{pq}$, the surrogate pair lies within
\begin{equation}\label{eq:Delta-pert}
   \Delta_{\mathrm{pert}}:=2C_gd^\kappa\delta
\end{equation}
of $(g_d,f)$ in the gauge~\eqref{eq:data-pert}. Let
\begin{equation}\label{eq:rho-delta}
  \rho:=\varepsilon(4K_\Pi d^{C_\Pi})^{-1},
  \qquad
  \delta:=\rho(2C_gd^\kappa)^{-1}
         =\varepsilon(8C_gK_\Pi d^{C_\Pi+\kappa})^{-1}.
\end{equation}
Then $\Delta_{\mathrm{pert}}=\rho$ and $\delta,\rho\in(0,1]$. Let $K$ be the universal constant from
\Cref{lem:prec}. \Cref{lem:mlp-cost}, applied with target $\rho$, parameter $\xi$, and $A=K$,
ensures that there exists $N=N(\rho)\in\N\cap[2,\infty)$ such that
\begin{equation}\label{eq:representation-factor}
  \frac{e^{N/2}(1+2LT)^N}{N^{N/2}}\le\rho,
  \qquad
  (KN)^N\le C_\xi\rho^{-(2+\xi)}.
\end{equation}
We set $M:=N$. In particular,
\begin{equation}\label{eq:representation-poly}
  (KM)^N\le C d^{C_\Pi(2+\xi)}\varepsilon^{-(2+\xi)}.
\end{equation}
Moreover,
\begin{equation}\label{eq:delta-lb}
  \delta\ge C^{-1}\varepsilon d^{-(C_\Pi+\kappa)}.
\end{equation}

Next, we establish the mean-square error estimate.
Let $U^{0,\delta}_{N,M}$ be the MLP estimator built from the surrogate data, as in
\Cref{prop:repr}; when \Cref{set:mlp} is instantiated with these data, this is the estimator
$U^{2,0}_{N,M}$. \Cref{cor:mlp-stab}, applied with perturbation level
$\Delta_{\mathrm{pert}}=\rho$, the identification $v_1(0,x)=u_d(T,x)$ from~\eqref{eq:reversal}
(the theorem hypothesizes $u_d$ in the class~\eqref{eq:reg-class}), and
\eqref{eq:representation-factor} ensure,
for every $x\in\R^d$, that
\[
  \big(\E[\abs{U^{0,\delta}_{N,M}(0,x)-u_d(T,x)}^2]\big)^{1/2}
  \le 2\rho\big(e^{LT}(T+1)\big)^{q+1}(B^q+1)\Pi_d(x).
\]
After squaring and integrating over $[0,1]^d$, \eqref{eq:KPi} and the definition of $\rho$ show that
\begin{equation}\label{eq:mean-square-root}
  \int_{[0,1]^d}\E[\abs{U^{0,\delta}_{N,M}(0,x)-u_d(T,x)}^2]\,dx
  \le\frac{\varepsilon^2}{4}<\varepsilon^2.
\end{equation}
The map $(x,\omega)\mapsto U^{0,\delta}_{N,M}(0,x)(\omega)$ is jointly measurable. Indeed, this
follows by induction in the finite recursion~\eqref{eq:mlp} from the continuity of the surrogate
realizations and the joint measurability of the Brownian motions and random times. Hence Tonelli's
theorem gives
\[
  \E\!\left[
    \int_{[0,1]^d}\abs{U^{0,\delta}_{N,M}(0,x)-u_d(T,x)}^2\,dx
  \right]
  =
  \int_{[0,1]^d}\E[\abs{U^{0,\delta}_{N,M}(0,x)-u_d(T,x)}^2]\,dx.
\]
Consequently, there exists $\omega^\ast\in\Omega$ such that
\begin{equation}\label{eq:goodomega}
  \int_{[0,1]^d}
  \abs{U^{0,\delta}_{N,M}(0,x)(\omega^\ast)-u_d(T,x)}^2\,dx<\varepsilon^2.
\end{equation}

We now represent the selected realization by a ResNet. \Cref{prop:repr} establishes the existence of a ResNet
$\Psi_{d,\varepsilon}:=\Psi^{\omega^\ast}_{N,M,\delta}$ whose realization equals the selected MLP
sample. Equation~\eqref{eq:goodomega} gives the error bound in~\eqref{eq:main}.

It remains to estimate the number of parameters.
For general surrogate families, \Cref{lem:prec} and~\eqref{eq:Bdd-hat} give
\[
  \Param(\Psi_{d,\varepsilon})
  \le C\widehat B_{d,\delta}(KM)^N,
  \qquad
  \widehat B_{d,\delta}
  \le C d^{\kappa+1}\delta^{-\kappa}+Cd^2.
\]
Combining this estimate with~\eqref{eq:delta-lb} and~\eqref{eq:representation-poly} yields
\[
  \Param(\Psi_{d,\varepsilon})
  \le C_0d^{\eta_d}\varepsilon^{-\eta_\varepsilon},
\]
where
\[
  \eta_d:=\max\!\big\{\kappa+1+\kappa C_\Pi+\kappa^2+C_\Pi(2+\xi),
                      2+C_\Pi(2+\xi)\big\},
  \qquad
  \eta_\varepsilon:=\kappa+2+\xi,
\]
and $C_0\in(0,\infty)$ is independent of $d$ and $\varepsilon$. For part~\textup{(i)}, let
$\xi=1$ and $\eta:=\max\{C_0,\eta_d,\eta_\varepsilon\}$. This proves the parameter estimate in
\eqref{eq:main}.

Under the hypotheses of part~(ii), \Cref{cor:trunc-surrogate,lem:ridge-surrogate} verify
\Cref{set:nn} with $\kappa=1$ and $p=1$, so part~(i) applies. Every profile satisfies
$\abs{\phi_j(v)}\le\abs{\phi_j(0)}+\mathrm{TV}(\phi_j)\le M_\phi+T_\phi$, so that
$\sup_{x\in\R^d}\abs{g_d(x)}\le J(M_\phi+T_\phi)$ and, by~\eqref{eq:ridge-sup},
$\sup_{x\in\R^d}\abs{(\cR_a\mathfrak g_{d,\delta})(x)}\le J(M_\phi+T_\phi)+1$. We therefore run the
preceding argument with the explicit common constants
\[
\begin{gathered}
  L_*:=\max\{L,L_f\},\qquad B_*:=\max\{B,B_f,J(M_\phi+T_\phi)+1\},\\[2pt]
  \kappa=1,\qquad p=1,\qquad q=\max\{2,\kappa/p\}=2,
\end{gathered}
\]
in place of $L$ and $B$; these depend only on the constants in the hypotheses and not on $d$,
$\varepsilon$, or the direction vectors $\alpha^d_j$. By the Gaussian moment
estimate following~\eqref{eq:KPi}, the exponent $C_\Pi=pq/2=1$ is then admissible, and we fix this
choice, with $K_\Pi$ depending only on $T$, $L_*$, and $B_*$. \Cref{lem:ridge-surrogate} supplies the uniform error and
linear cost bounds in~\eqref{eq:ridge-bounds}. By the first estimate
in~\eqref{eq:ridge-bounds} and the $f$-part of~\eqref{eq:nn-acc}, both surrogate errors are
bounded by $\bar\delta$ uniformly, and the gauge weight $(1+\norm x)^{pq}+\abs v^q$
in~\eqref{eq:data-pert} is at least $1$. Hence the surrogate pair has perturbation level
\[
  \bar\Delta_{\mathrm{pert}}:=\bar\delta.
\]
Choose
\[
  \bar\delta:=\rho=\frac{\varepsilon}{4K_\Pi d^{C_\Pi}}.
\]
Then $\bar\Delta_{\mathrm{pert}}=\rho$ and $\bar\delta\in(0,1]$, so the preceding MLP error
estimate, selection argument, and representation argument apply without change. Furthermore,
\Cref{cor:trunc-surrogate}, \eqref{eq:ridge-bounds}, and~\eqref{eq:Bdd} give
\[
\begin{aligned}
  B_{d,\bar\delta}
  &\le C\bigl(d\bar\delta^{-1}+d^2\bigr)
  \le C\bigl(d^{1+C_\Pi}\varepsilon^{-1}+d^2\bigr)
  \le C d^{1+C_\Pi}\varepsilon^{-1}.
\end{aligned}
\]
In the last step we used $C_\Pi\ge1$, $d\ge1$, and $\varepsilon\le1$.
Multiplying this estimate by~\eqref{eq:representation-poly} yields
\[
  \Param(\Psi_{d,\varepsilon})
  \le C_\xi
  d^{1+C_\Pi(3+\xi)}
  \varepsilon^{-(3+\xi)}
  = C_\xi\,
  d^{4+\xi}
  \varepsilon^{-(3+\xi)},
\]
since $C_\Pi=1$. This proves~\eqref{eq:main-sharp}.
\end{proof}

Since $[0,1]^d$ has Lebesgue measure one, H\"older's inequality shows that the networks of
\Cref{thm:main} also satisfy
$[\int_{[0,1]^d}\abs{u_d(T,x)-(\Real_a\Psi_{d,\varepsilon})(x)}^r\,dx]^{1/r}\le\varepsilon$ for
every $r\in(0,2]$.

\section{Consequences and examples}\label{sec:advantages}

\begin{remark}[An Allen--Cahn-type example]\label{rem:allen-cahn-example}
Let $a$ be an admissible sigmoidal activation  in the sense of
\Cref{lem:1d-approx}, let
\[
  g_d(x)=\tanh(\langle\alpha^d,x\rangle),
  \qquad \alpha^d\in\R^d\ \text{arbitrary},
\]
and define the globally Lipschitz truncation
\[
  \widetilde f(r)=
  \begin{cases}
    r-r^3, & \abs{r}\le 1,\\
    0,     & \abs{r}>1.
  \end{cases}
\]
Then \Cref{lem:ridge-surrogate} applies with $J=1$, $L_\phi=1$, $M_\phi=0$, and
$T_\phi=2$. Since $g_d$ takes values in $[-1,1]$ and the constants $\pm1$ are solutions
of~\eqref{eq:pde} because $\widetilde f(\pm1)=0$, the parabolic
comparison principle implies that the corresponding solution remains in this interval.
Consequently, $\widetilde f$ agrees with the classical Allen--Cahn reaction $r\mapsto r-r^3$ on
the range of the solution, and the truncation level is independent of $d$.

\Cref{thm:main}(ii) therefore yields the bound~\eqref{eq:main-sharp} for every $\xi>0$.
\end{remark}

\begin{remark}[General measures]
With $p$, $q$, $\Pi_d$, and $C_\Pi$ as in the proof of \Cref{thm:main}, the Lebesgue measure on
$[0,1]^d$ in~\eqref{eq:main} may be replaced by Borel probability measures $\nu_d$ with $\int_{\R^d}(1+\norm{x})^{2pq}\,\nu_d(dx)\le C_0d^\gamma$ for some
$C_0\in(0,\infty)$ and $\gamma\in[0,\infty)$, the $L^2([0,1]^d)$-error being replaced by the
$L^2(\nu_d)$-error: since the Gaussian moment term in $\Pi_d$ is of order $d^{1/2}$, one has
$\int_{\R^d}\Pi_d(x)^2\,\nu_d(dx)\le Cd^{\max\{\gamma,pq\}}$, and the proof applies with
$C_\Pi=\frac12\max\{\gamma,pq\}$. In particular, the exponents of part~(ii) persist for
$\gamma\le2$.
\end{remark}

\section{Conclusion}\label{sec:conclusion}

\Cref{thm:main} establishes a deterministic, curse-of-dimensionality-free ResNet expressivity
result for semilinear heat equations whose data admit the network surrogates of \Cref{set:nn};
for admissible sigmoidal activations, globally Lipschitz truncations, and the ridge-sum data of
\Cref{lem:ridge-surrogate}, it yields the bound $C_\xi d^{4+\xi}\varepsilon^{-(3+\xi)}$ for every
$\xi>0$. The key ingredient is the accumulator construction of \Cref{prop:repr}, which represents
one realization of an MLP estimator by concatenating residual updates, so that parameter counts
add and no identity representation of the activation is needed. Extensions to state-dependent
coefficients, gradient-dependent nonlinearities, and space--time approximations, as well as
sharper exponents and constructive training procedures, remain open.

\appendix

\section{Proofs of the quantitative MLP inputs}\label{app:mlp-proofs}

In this appendix we prove \Cref{lem:apriori,lem:pde-stab,cor:mlp-stab,lem:mlp-cost} and record
the constants used in \Cref{thm:main}.

\begin{proof}[Proof of \Cref{lem:apriori}]
Fix $i\in\{1,2\}$ and $x\in\R^d$, write $v=v_i$, $g=g_i$, $f=f_i$, and
$\norm{Z}_{L^q}:=(\E[\abs{Z}^q])^{1/q}$ for a random variable $Z$, and let
\begin{equation}\label{eq:apriori-PhiR}
  \Phi(t):=\sup_{s\in[t,T]}\norm{v(s,x+\mathbf W_s)}_{L^q},\qquad
  R:=\sup_{s\in[0,T]}\big(\E[(1+\norm{x+\mathbf W_s})^{pq}]\big)^{1/q},
\end{equation}
so that the assertion reads $\Phi(0)\le e^{LT}(T+1)BR$. The Gaussian moments of $\mathbf W$ and the
polynomial growth of $v$ give $R<\infty$ and $\Phi(0)<\infty$, and $R\ge1$ since
$1+\norm{x+\mathbf W_s}\ge1$; moreover, $\Phi$ is
nonincreasing, hence measurable, with $\int_0^T\Phi(r)\,dr\le T\Phi(0)<\infty$.

Fix $t\in[0,T]$. Evaluating~\eqref{eq:sfp} at the random point $x+\mathbf W_t$ gives,
$\PR$-almost surely,
\begin{equation}\label{eq:apriori-sfp-shift}
  v(t,x+\mathbf W_t)
  =\E\Big[g\big(x+\mathbf W_t+\widetilde{\mathbf W}_{T-t}\big)
    +\int_t^T f\big(v(r,x+\mathbf W_t+\widetilde{\mathbf W}_{r-t})\big)\,dr\,\Big|\,\mathbf W_t\Big],
\end{equation}
where $\widetilde{\mathbf W}$ is an independent copy of the Brownian motion and the conditional
expectation is taken over $\widetilde{\mathbf W}$ only; the patched process
$(\mathbf W_t+\widetilde{\mathbf W}_{r-t})_{r\in[t,T]}$ has the same law as
$(\mathbf W_r)_{r\in[t,T]}$. Taking $L^q$-norms in~\eqref{eq:apriori-sfp-shift} and using
conditional Jensen's inequality $\|\E[Z\mid\mathbf W_t]\|_{L^q}\le\|Z\|_{L^q}$, this law identity,
the triangle inequality, and Minkowski's integral inequality yield
\begin{equation}\label{eq:apriori-split}
  \norm{v(t,x+\mathbf W_t)}_{L^q}
  \le \big\|\,g(x+\mathbf W_T)\,\big\|_{L^q}
     +\int_t^T\norm{f\big(v(r,x+\mathbf W_r)\big)}_{L^q}\,dr.
\end{equation}
By the growth bound $\abs{g(y)}\le B(1+\norm y)^p$ and the definition of $R$,
\begin{equation}\label{eq:apriori-g}
  \norm{g(x+\mathbf W_T)}_{L^q}
  \le B\big(\E\big[(1+\norm{x+\mathbf W_T})^{pq}\big]\big)^{1/q}
  \le B\,R,
\end{equation}
and since $\abs{f(v)}\le\abs{f(0)}+L\abs{v}\le B+L\abs{v}$, for every $r\in[t,T]$,
\begin{equation}\label{eq:apriori-f}
  \norm{f\big(v(r,x+\mathbf W_r)\big)}_{L^q}\le B+L\,\Phi(r).
\end{equation}
Inserting~\eqref{eq:apriori-g} and~\eqref{eq:apriori-f} into~\eqref{eq:apriori-split} and using
$R\ge1$ gives
\begin{equation}\label{eq:apriori-pre-gronwall}
  \norm{v(t,x+\mathbf W_t)}_{L^q}
  \le B\,R + BT + L\!\int_t^T\!\Phi(r)\,dr
  \le (T+1)B\,R + L\!\int_t^T\!\Phi(r)\,dr.
\end{equation}
The right-hand side is nondecreasing as $t$ decreases, so taking the supremum over
$t\in[\tau,T]$ gives, for every $\tau\in[0,T]$,
\begin{equation}\label{eq:apriori-Phi-ineq}
  \Phi(\tau)\le (T+1)B\,R + L\!\int_\tau^T\!\Phi(r)\,dr.
\end{equation}
Applying the backward Gr\"onwall inequality to~\eqref{eq:apriori-Phi-ineq}, with $\alpha=(T+1)BR$
and $\tau=0$, yields $\Phi(0)\le e^{LT}(T+1)BR$, which is the assertion.
\end{proof}

\begin{proof}[Proof of \Cref{lem:pde-stab}]
Fix $x\in\R^d$ and let
\begin{equation}\label{eq:pde-stab-Pi}
  \Pi(x):=\big(1+\norm{x}+(\E[\norm{\mathbf W_T}^{pq}])^{1/(pq)}\big)^{pq},
  \qquad
  e(t):=\sup_{s\in[t,T]}\E\big[\,\abs{v_1(s,x+\mathbf W_s)-v_2(s,x+\mathbf W_s)}\,\big].
\end{equation}
By \Cref{lem:apriori} and Jensen's inequality, $e(0)<\infty$, and $e$ is nonincreasing, hence
measurable with $\int_0^Te(s)\,ds\le Te(0)<\infty$. It suffices to prove the stronger bound
$e(0)\le\Delta(e^{LT}(T+1))^{q+1}(B^q+1)\Pi(x)$.

Fix $t\in[0,T]$ and $y\in\R^d$, and let $\widetilde{\mathbf W}$ be an independent copy of
$\mathbf W$. Subtracting the fixed-point equation~\eqref{eq:sfp} for $v_2$ from that for $v_1$ at $(t,y)$ and
adding and subtracting $f_2(v_1)$ inside the integrand gives
\begin{equation}\label{eq:pde-stab-split}
  v_1(t,y)-v_2(t,y)
  =\E\big[(g_1-g_2)(y+\widetilde{\mathbf W}_{T-t})\big]
   +\E\!\int_t^T\!\underbrace{\big[f_1(v_1)-f_2(v_1)\big]}_{\text{data part}}
   +\underbrace{\big[f_2(v_1)-f_2(v_2)\big]}_{\text{Lipschitz part}}\,ds,
\end{equation}
where $f_i(v_j)$ abbreviates $f_i(v_j(s,y+\widetilde{\mathbf W}_{s-t}))$. Let
\begin{equation}\label{eq:pde-stab-Lambda}
  \Lambda:=\sup_{s\in[0,T]}\E\big[(1+\norm{x+\mathbf W_s})^{pq}+\abs{v_1(s,x+\mathbf W_s)}^q\big],
\end{equation}
which is finite by the Gaussian moments of $\mathbf W$ and \Cref{lem:apriori}; see
\eqref{eq:pde-stab-Lambda-bnd} below. The perturbation bound~\eqref{eq:data-pert}, applied at
$z=y+\widetilde{\mathbf W}_{T-t}$, $v=0$ for the boundary term and at
$z=y+\widetilde{\mathbf W}_{s-t}$, $v=v_1(s,z)$ for the data part, gives
\[
\begin{gathered}
  \abs{(g_1-g_2)(y+\widetilde{\mathbf W}_{T-t})}\le\Delta(1+\norm{y+\widetilde{\mathbf W}_{T-t}})^{pq},\\[2pt]
  \abs{f_1(v_1)-f_2(v_1)}\le\Delta\big((1+\norm{y+\widetilde{\mathbf W}_{s-t}})^{pq}+\abs{v_1}^q\big).
\end{gathered}
\]
Conditioning on $\mathbf W_t$ and setting $y=x+\mathbf W_t$, independence and stationarity of the
increments give $x+\mathbf W_t+\widetilde{\mathbf W}_{s-t}\stackrel{d}{=}x+\mathbf W_s$, so that,
after taking expectations, the boundary term in~\eqref{eq:pde-stab-split} is at most
$\Delta\Lambda$ and the data integrand is at most $\Delta\Lambda$ for every $s$. Hence, with
$d(r):=\E[\abs{v_1(r,x+\mathbf W_r)-v_2(r,x+\mathbf W_r)}]$, conditioning on $\mathbf W_r$,
applying~\eqref{eq:pde-stab-split} with initial time $r$ and $y=x+\mathbf W_r$, and using
$\Lip(f_2)\le L$ yields
$d(r)\le[1+(T-r)]\Delta\Lambda+L\int_r^Td(s)\,ds\le(T+1)\Delta\Lambda+L\int_r^Td(s)\,ds$ for
every $r\in[0,T]$.
Since $d\le e$ and $\int_r^Te\le\int_t^Te$ for $r\ge t$, taking the supremum over $r\in[t,T]$
gives
\begin{equation}\label{eq:pde-stab-gronwall}
  e(t)\le (T+1)\,\Delta\,\Lambda + L\!\int_t^T e(s)\,ds
  \qquad(t\in[0,T]),
\end{equation}
and the backward Gr\"onwall inequality of the proof of \Cref{lem:apriori} yields
\begin{equation}\label{eq:pde-stab-e0}
  e(0)\le (T+1)\,\Delta\,\Lambda\,e^{LT}.
\end{equation}
To bound $\Lambda$, note that $\E[(1+\norm{x+\mathbf W_s})^{pq}]\le\Pi(x)$ for every
$s\in[0,T]$, by the triangle inequality, Minkowski's inequality in $L^{pq}(\Omega,\PR)$, and
$\E[\norm{\mathbf W_s}^{pq}]\le\E[\norm{\mathbf W_T}^{pq}]$, while \Cref{lem:apriori} applied to
$v_1$ gives $\E[\abs{v_1(s,x+\mathbf W_s)}^q]\le(e^{LT}(T+1)B)^q\,\Pi(x)$ for every $s\in[0,T]$;
hence
\begin{equation}\label{eq:pde-stab-Lambda-bnd}
  \Lambda\le\big((e^{LT}(T+1)B)^q+1\big)\,\Pi(x).
\end{equation}
Inserting~\eqref{eq:pde-stab-Lambda-bnd} into~\eqref{eq:pde-stab-e0} and using
$(e^{LT}(T+1)B)^q+1\le(e^{LT}(T+1))^{q}(B^q+1)$, which holds because $e^{LT}(T+1)\ge1$, yields
the claimed bound.
\end{proof}

\begin{proof}[Proof of \Cref{cor:mlp-stab}]
Fix $x\in\R^d$, write $U:=U^{2,0}_{N,M}(0,x)$, and let $\Pi(x)$ be as in~\eqref{eq:pde-stab-Pi}.
Inserting the deterministic number $v_2(0,x)$ and using the triangle inequality in
$L^2(\Omega,\PR)$ gives
\begin{equation}\label{eq:mlp-stab-triangle}
  \big(\E[\abs{U-v_1(0,x)}^2]\big)^{1/2}
  \le \underbrace{\big(\E[\abs{U-v_2(0,x)}^2]\big)^{1/2}}_{\text{MLP approximation error}}
  +\underbrace{\abs{v_2(0,x)-v_1(0,x)}}_{\text{bias from surrogates}}.
\end{equation}
Since $v_1$ and $v_2$ solve~\eqref{eq:sfp} with data differing by at most $\Delta$ in the sense
of~\eqref{eq:data-pert}, \Cref{lem:pde-stab} at $t=0$, where $\mathbf W_0=0$, gives
\begin{equation}\label{eq:mlp-stab-bias}
  \abs{v_2(0,x)-v_1(0,x)}\le\Delta\,(e^{LT}(T+1))^{q+1}(B^q+1)\,\Pi(x).
\end{equation}
The estimator $U^{2,0}_{N,M}$ is the full-history recursive MLP estimator for the surrogate
solution $v_2$, and the MLP convergence estimate \cite[Theorem~3.5]{HJKNW2018}, in the form of
\cite[Corollary~2.4]{HJKN2020}, controls its full $L^2$-error under the Lipschitz and
$q$th-moment hypotheses of \Cref{set:mlp}:
\begin{equation}\label{eq:mlp-stab-mc}
  \big(\E[\abs{U^{2,0}_{N,M}(0,x)-v_2(0,x)}^2]\big)^{1/2}
  \le \frac{e^{M/2}\,(1+2LT)^{N}}{M^{N/2}}\,(e^{LT}(T+1))^{q+1}(B^q+1)\,\Pi(x).
\end{equation}
In~\eqref{eq:mlp-stab-mc}, the factor $\Pi(x)$ is obtained by bounding the terminal-data and
$f_2(0)$ terms, \Cref{lem:apriori} verifying the required $L^2$-integrability;
in~\eqref{eq:mlp-stab-bias}, it arises from the Gaussian growth term together with the moment
bound of \Cref{lem:apriori} applied to $v_1$; see~\eqref{eq:pde-stab-Lambda-bnd}. Inserting~\eqref{eq:mlp-stab-bias} and~\eqref{eq:mlp-stab-mc}
into~\eqref{eq:mlp-stab-triangle} and collecting the common factor
$(e^{LT}(T+1))^{q+1}(B^q+1)\,\Pi(x)$ gives~\eqref{eq:mlp-stab}. The detailed tracking of the
constants in~\eqref{eq:mlp-stab-mc} is carried out in \cite[Corollary~2.4]{HJKN2020}.
\end{proof}

\begin{proof}[Proof of \Cref{lem:mlp-cost}]
Set $\kappa_0:=e^{1/2}(1+2LT)$ and $E_n:=(\kappa_0/\sqrt n)^n$ for $n\in\N$, so that
\eqref{eq:mlp-target} reads $E_N\le\varepsilon$. Since $E_n\to0$ and
\[
  \frac{n\log(An)}{-\log E_{n-1}}
  =\frac{n\log(An)}{(n-1)\big(\tfrac12\log(n-1)-\log\kappa_0\big)}
  \longrightarrow 2
  \qquad(n\to\infty),
\]
there exists $N_0\in\N\cap[2,\infty)$, depending only on $L,T,A,\xi$, such that
$-\log E_{n-1}>0$ and this quotient is at most $2+\xi$ for every $n>N_0$. Let $N\ge N_0$ be
minimal with $E_N\le\varepsilon$. If $N>N_0$, then $E_{N-1}>\varepsilon$ by minimality, and
hence
\[
  (AN)^N=\exp\big(N\log(AN)\big)\le E_{N-1}^{-(2+\xi)}<\varepsilon^{-(2+\xi)}.
\]
If $N=N_0$, then $(AN)^N=(AN_0)^{N_0}$ is a constant depending only on $L,T,A,\xi$, which is
absorbed into $C_{\xi,A}$. This establishes \eqref{eq:mlp-target}--\eqref{eq:mlp-balance-cost}.
\end{proof}

\section{Proof of the one-dimensional approximation lemma}\label{app:1d-proof}

We carry out the proof of \Cref{lem:1d-approx}. Steps~1--3 are a classical sigmoidal staircase
argument with bandwidth equal to the grid spacing; Step~4 proves the finite-total-variation
assertion by a variation-adapted staircase with a separated-center bandwidth.

\begin{proof}[Proof of \Cref{lem:1d-approx}]
	Since $a$ is nondecreasing, its normalization $\widetilde a$ is nondecreasing with limits $0$ and
	$1$. Consequently,
	\begin{equation}\label{eq:a-normalized}
		0\le\widetilde a\le1,\qquad \widetilde a'\ge0,\qquad
		\norm{\widetilde a'}_{L^1(\R)}=\mathrm{TV}(\widetilde a)=1.
	\end{equation}
	
	\smallskip
	\emph{Step 1: The staircase target.}
	Fix $R\in[1,\infty)$ and $n\in\N$, and set $\eta:=2R/n$. For $k\in\{0,1,\dots,n\}$ put
	$v_k:=-R+k\eta$. Define the
	\emph{midpoint staircase} $S_n\colon\R\to\R$ by
	\begin{equation}\label{eq:staircase}
		S_n(v):=\phi(v_0)+\sum_{k=1}^{n}\Delta_k\,\mathbf{1}_{(c_k,\infty)}(v),\qquad
		\Delta_k:=\phi(v_k)-\phi(v_{k-1}),\quad c_k:=\tfrac{v_{k-1}+v_k}{2}.
	\end{equation}
	By construction $S_n(v)=\phi(v_k)$ for $v\in(c_k,c_{k+1})$ with $c_0:=-\infty$, $c_{n+1}:=+\infty$,
	and the staircase is constant outside its central region: $S_n(v)=\phi(v_0)$ for $v\le c_1$ and
	$S_n(v)=\phi(v_n)$ for $v>c_n$. At each jump point $c_k$, one has
	$S_n(c_k)=\phi(v_{k-1})$ and $\abs{c_k-v_{k-1}}=R/n$. Since $\phi$ is
	$\Lip(\phi)$-Lipschitz and $\abs{v-v_k}\le R/n$ for $v\in[-R,R]\cap(c_k,c_{k+1})$ and every
	$k\in\{0,\dots,n\}$,
	\begin{equation}\label{eq:staircase-err}
		\sup_{v\in[-R,R]}\abs{\phi(v)-S_n(v)}\le \Lip(\phi)\frac{R}{n},
		\qquad
		\abs{\Delta_k}\le\Lip(\phi)\eta.
	\end{equation}
	If $\phi$ is constant on $(-\infty,-R]$ and on $[R,\infty)$, then the same estimate holds with
	$[-R,R]$ replaced by $\R$, because $S_n(v)=\phi(v)$ for $v\notin[-R,R]$.
	
	\smallskip
	\emph{Step 2: Sigmoidal smoothing.}
	We use the elementary sampling estimate
	\begin{equation}\label{eq:amalgam}
		\sup_{x\in\R}\sum_{j\in\Z}\abs{g(x-j)}
		\le\norm{g}_{L^1(\R)}+\mathrm{TV}(g)
	\end{equation}
	for every integrable function $g\colon\R\to\R$ whose given pointwise representative has bounded
	variation. Indeed, by
	periodicity of the sum it suffices to take $x\in[0,1)$. For every $m\in\Z$,
	\[
	\abs{g(x+m)}
	\le\int_m^{m+1}\abs{g(r)}\,dr+\mathrm{TV}\bigl(g;[m,m+1)\bigr),
	\]
	and summing over $m$ proves~\eqref{eq:amalgam}.
	
	Let $H:=\mathbf 1_{(0,\infty)}$ (so $H(0)=0$) and $\psi_a:=\abs{\widetilde a-H}$. Since the absolute-value map is
	$1$-Lipschitz, \eqref{eq:a-normalized} gives
	\begin{equation}\label{eq:psi-variation}
		\norm{\psi_a}_{L^1(\R)}=\widetilde\Lambda_a,
		\qquad
		\mathrm{TV}(\psi_a)
		\le\mathrm{TV}(\widetilde a-H)
		\le\mathrm{TV}(\widetilde a)+\mathrm{TV}(H)=2.
	\end{equation}
	
	Set the bandwidth $h:=\eta=2R/n$ and define
	\begin{equation}\label{eq:smoothed-stair}
		Q_{n,R}(v):=\phi(v_0)+\sum_{k=1}^{n}\Delta_k\,
		\widetilde a\!\left(\frac{v-c_k}{\eta}\right).
	\end{equation}
	Since $c_k=c_1+(k-1)\eta$, writing $x:=(v-c_1)/\eta$ and applying
	\eqref{eq:amalgam} to $\psi_a$ yield, for every $v\in\R$,
	\[
	\begin{aligned}
		\abs{Q_{n,R}(v)-S_n(v)}
		&\le\sum_{k=1}^n\abs{\Delta_k}\,\psi_a(x-(k-1))\\
		&\le\Lip(\phi)\eta\sum_{j\in\Z}\psi_a(x-j)
		\le\Lip(\phi)\eta\bigl(\widetilde\Lambda_a+2\bigr).
	\end{aligned}
	\]
	Together with~\eqref{eq:staircase-err}, this proves
	\begin{equation}\label{eq:smoothed-error}
		\sup_{v\in[-R,R]}\abs{\phi(v)-Q_{n,R}(v)}
		\le\bigl(5+2\widetilde\Lambda_a\bigr)\Lip(\phi)\frac{R}{n}.
	\end{equation}
	If $\phi$ is constant on the two exterior intervals appearing in the lemma, the global version of
	\eqref{eq:staircase-err} established in Step~1 and the pointwise estimate above give the same bound
	with $[-R,R]$ replaced by $\R$.
	
	The function $Q_{n,R}$ is continuously differentiable. Applying \eqref{eq:amalgam} to the
	continuous representative of $\widetilde a'$ gives
	\[
	\begin{aligned}
		\abs{Q_{n,R}'(v)}
		&\le\sum_{k=1}^n\frac{\abs{\Delta_k}}{\eta}
		\abs{\widetilde a'(x-(k-1))}\\
		&\le\Lip(\phi)\sum_{j\in\Z}\abs{\widetilde a'(x-j)}
		\le\Lip(\phi)\bigl(1+\mathrm{TV}(\widetilde a')\bigr).
	\end{aligned}
	\]
	Consequently,
	\begin{equation}\label{eq:smoothed-lip}
		\Lip(Q_{n,R})\le\bigl(1+\mathrm{TV}(\widetilde a')\bigr)\Lip(\phi).
	\end{equation}
	
	\smallskip
	\emph{Step 3: Realization as an FNN.}
	Define $\phi_{n,R}:=\big((W_1,B_1),(W_2,B_2)\big)\in\bN$ with $W_1\in\R^{n\times1}$,
	$B_1\in\R^n$, $W_2\in\R^{1\times n}$, $B_2\in\R$ given by
	\begin{equation}\label{eq:1d-entries}
		(W_1)_{k,1}:=\frac{1}{\eta},\qquad
		(B_1)_{k}:=-\frac{c_k}{\eta},\qquad
		(W_2)_{1,k}:=\frac{\Delta_k}{A_+-A_-},\qquad
		B_2:=\phi(v_0)-\frac{A_-}{A_+-A_-}\sum_{j=1}^n\Delta_j,
	\end{equation}
	for $k\in\{1,\dots,n\}$. Since $\widetilde a=(a-A_-)/(A_+-A_-)$, \Cref{def:fnn} gives
	$(\cR_a\phi_{n,R})(v)=\phi(v_0)+\sum_{k=1}^{n}\Delta_k\,\widetilde a((v-c_k)/\eta)=Q_{n,R}(v)$
	for every $v\in\R$, so that $\cD(\phi_{n,R})=(1,n,1)$ and
	$\cP(\phi_{n,R})=2n+(n+1)=3n+1$.
	Finally, \eqref{eq:smoothed-error} and \eqref{eq:smoothed-lip} give
	\eqref{eq:1d-approx-in}--\eqref{eq:1d-lip}, as well as \eqref{eq:1d-approx-global}, with, for
	example,
	\[
	C_a:=5+2\widetilde\Lambda_a+\mathrm{TV}(\widetilde a').
	\]

	\smallskip
	\emph{Step 4: The finite-total-variation assertion.}
	The notation of Steps~1--3 is not used in this step. Let $\phi$ have finite total variation
	$T:=\mathrm{TV}(\phi)$, let $\delta\in(0,1]$, and set $\lambda:=\delta/3$ and
	\[
	\varpi_a(t):=\sup_{\abs r\ge t}\,\abs{\widetilde a(r)-\mathbf 1_{(0,\infty)}(r)}
	\qquad(t\in(0,\infty)),
	\]
	so that $\varpi_a(t)\to0$ as $t\to\infty$ by~\eqref{eq:a-sigmoidal}. Since $\phi$ is continuous
	with finite total variation, the finite limits $\lim_{u\to\pm\infty}\phi(u)$ exist. Let
	$V(v):=\mathrm{TV}(\phi;(-\infty,v])$; by additivity of the total variation,
	$\abs{\phi(v)-\lim_{u\to\infty}\phi(u)}\le\mathrm{TV}(\phi;[v,\infty))=T-V(v)$ for every
	$v\in\R$. If $T\le\lambda$, then $\abs{\phi(v)-\lim_{u\to\infty}\phi(u)}\le\lambda\le\delta$ for
	every $v\in\R$, and the network of architecture $(1,1,1)$ with zero weight matrices, zero hidden
	bias, and output bias $\lim_{u\to\infty}\phi(u)$ satisfies~\eqref{eq:1d-bv}. Assume now
	$T>\lambda$; then $\phi$ is not constant, so $L:=\Lip(\phi)\in(0,\infty)$. The function $V$ is
	nondecreasing with $\lim_{v\to-\infty}V(v)=0$ and $\lim_{v\to\infty}V(v)=T$; moreover, for
	$x<y$, additivity gives $0\le V(y)-V(x)=\mathrm{TV}(\phi;[x,y])\le L\,(y-x)$, so $V$ is
	$L$-Lipschitz and, in particular, continuous. Set $K:=\lceil T/\lambda\rceil\ge2$ and, for
	$k\in\{1,\dots,K-1\}$, let
	\[
	c_k:=\inf\{v\in\R:V(v)\ge k\lambda\}.
	\]
	Since $k\lambda\le(K-1)\lambda<T$ and $V$ is continuous with the stated limits, each $c_k$ is a
	real number with $V(c_k)=k\lambda$, and $c_1<c_2<\dots<c_{K-1}$, since $c_{k+1}\ge c_k$ and
	$V(c_{k+1})=(k+1)\lambda\neq k\lambda=V(c_k)$. Moreover, the Lipschitz bound for $V$ gives, for
	every $k\in\{1,\dots,K-2\}$,
	\begin{equation}\label{eq:gap}
		\lambda=V(c_{k+1})-V(c_k)\le L\,(c_{k+1}-c_k),
		\qquad\text{hence}\qquad
		c_{k+1}-c_k\ge\lambda/L.
	\end{equation}
	Consider the cells $I_0:=(-\infty,c_1]$, $I_k:=(c_k,c_{k+1}]$ for $k\in\{1,\dots,K-2\}$, and
	$I_{K-1}:=(c_{K-1},\infty)$, and the values $s_0:=\phi(c_1)$, $s_k:=\phi(c_{k+1})$ for
	$k\in\{1,\dots,K-2\}$, and $s_{K-1}:=\lim_{u\to\infty}\phi(u)$. The variation of $\phi$ over the
	closure of each cell is at most $\lambda$: it equals $V(c_1)=\lambda$ for $I_0$, equals
	$V(c_{k+1})-V(c_k)=\lambda$ for the interior cells, and is $T-V(c_{K-1})=T-(K-1)\lambda\le\lambda$
	for $I_{K-1}$. For $v\in I_k$ with $k\in\{0,\dots,K-2\}$, both $v$ and the sample point defining
	$s_k$ lie in the closure of $I_k$, so $\abs{\phi(v)-s_k}\le\lambda$; for $v\in I_{K-1}$, the tail
	bound above gives $\abs{\phi(v)-s_{K-1}}\le T-V(v)\le T-V(c_{K-1})\le\lambda$. Hence the
	staircase
	\[
	S(v):=s_0+\sum_{k=1}^{K-1}\Delta_k\,\mathbf 1_{(c_k,\infty)}(v),
	\qquad
	\Delta_k:=s_k-s_{k-1},
	\]
	satisfies $\sup_{v\in\R}\abs{\phi(v)-S(v)}\le\lambda$. Moreover, $\abs{\Delta_k}\le\lambda$ for
	every $k$ by the cellwise variation bounds, while $\sum_{k=1}^{K-1}\abs{\Delta_k}\le T$ follows
	from the definition of the total variation applied to the ordered sample points
	$c_1<\dots<c_{K-1}$, followed by passage to the limit at $+\infty$.

	Since $\varpi_a(t)\to0$ as $t\to\infty$, we may choose $\theta\in(0,1]$, depending only on $a$
	and $T/\lambda$, such that $T\,\varpi_a(1/(2\theta))\le\lambda$. Set $h:=\theta\lambda/L$ and
	\[
	Q(v):=s_0+\sum_{k=1}^{K-1}\Delta_k\,\widetilde a\big((v-c_k)/h\big)
	\qquad(v\in\R).
	\]
	As in Step~3, $Q$ is the realization of an FNN $\phi_\delta$ of architecture $(1,K-1,1)$ with
	first-layer weights $1/h$, first-layer biases $-c_k/h$, output weights $\Delta_k/(A_+-A_-)$, and
	output bias $s_0-\sum_{k=1}^{K-1}\Delta_kA_-/(A_+-A_-)$. We estimate
	$\abs{Q(v)-S(v)}\le\sum_{k=1}^{K-1}\abs{\Delta_k}\,\psi\big((v-c_k)/h\big)$, where
	$\psi:=\abs{\widetilde a-\mathbf 1_{(0,\infty)}}$ satisfies $\psi\le1$ by~\eqref{eq:a-normalized}
	and $\psi(r)\le\varpi_a(\abs r)$ for $r\neq0$. Fix $v\in\R$ and let $k^\ast$ minimize
	$\abs{v-c_k}$ over $k$. By~\eqref{eq:gap}, distinct centers are at least $\lambda/L$ apart, and
	$\abs{v-c_k}\ge\tfrac12\abs{c_k-c_{k^\ast}}$ for $k\neq k^\ast$ by the choice of $k^\ast$; hence
	$\abs{v-c_k}\ge\lambda/(2L)$ and $\abs{v-c_k}/h\ge1/(2\theta)$ for every $k\neq k^\ast$. Consequently,
	\[
	\abs{Q(v)-S(v)}
	\le\abs{\Delta_{k^\ast}}
	+\Big(\sum_{k\neq k^\ast}\abs{\Delta_k}\Big)\varpi_a\big(1/(2\theta)\big)
	\le\lambda+T\,\varpi_a\big(1/(2\theta)\big)
	\le2\lambda.
	\]
	Together with the staircase bound, $\sup_{v\in\R}\abs{\phi(v)-(\cR_a\phi_\delta)(v)}\le3\lambda=\delta$,
	which is~\eqref{eq:1d-bv} with $m=K-1\le\lceil3T/\delta\rceil-1$.
\end{proof}

\section*{Acknowledgements}
This project has been partially supported by DFG – Project-ID 499552394 – SFB 1597.

\section*{Declaration of generative AI and AI-assisted technologies	in the manuscript preparation process}

The key ideas of the statements and the proofs of the main results of
this work are due to the authors. 
During the preparation of this work, the authors used OpenAI's
ChatGPT and the AI assistant Claude (Anthropic) to improve the language and readability of the manuscript, to check the consistency of notation, cross-references, to verify bibliographic details against the cited sources, and to flag possible issues
concerning mathematical
expositions.  All suggestions and proposed revisions were critically
reviewed by the authors. The authors independently verified all
mathematical statements, proofs, and references, made all final
decisions concerning the manuscript, and take full responsibility
for the content of the published article.

\bibliographystyle{acm}
\bibliography{refs}

\end{document}